\documentclass[11pt]{amsart}

\usepackage{amsfonts,amssymb,amsmath,mathrsfs,graphicx}
\usepackage{float}
\usepackage{epsfig}

\title[Stability of planar rarefaction wave to 2D Navier-Stokes]{Stability of the planar rarefaction wave to two-dimensional compressible Navier-Stokes equations}

\author[Li]{Linan Li}
\address[Linan Li]{\newline Institute of Applied Mathematics, AMSS, Academia Sinica, Beijing 100190, P. R. China}
\email{linanli@amss.ac.cn}

\author[Wang]{Yi Wang}
\address[Yi Wang]{\newline \newline CEMS, HCMS, NCMIS, Academy of Mathematics and Systems Science, Chinese Academy of Sciences, Beijing 100190, China
\newline and School of Mathematical Sciences, University of Chinese Academy of Sciences, Beijing 100049, China
}
\email{wangyi@amss.ac.cn}

\thanks{\textbf{Acknowledgment.} Y. Wang is supported by NSFC grants No. 11671385 and 11688101 and CAS Interdisciplinary Innovation Team}

\newtheorem{theorem}{Theorem}[section]
\newtheorem{lemma}{Lemma}[section]

\newtheorem{proposition}{Proposition}[section]
\newtheorem{remark}{Remark}[section]

\newcommand{\bbr}{\mathbb R}

\newcommand{\bbt}{\mathbb T}

\def\charf {\mbox{{\text 1}\kern-.30em {\text l}}}
\def\lam{\lambda}  % the "Boltzmann" scaling
\def\di{\displaystyle}

\begin{document}
%%%%%%%%%%%%%%%%

\date{\today}

  %\subjclass{92D25,74A25,76N10}
  %
  %\keywords{pressureless gas, Cucker-Smale particle, flocking dissipation. variational approach}
  %
  %\thanks{\textbf{Acknowledgment.} The work of S.-Y. Ha is partially supported by NRF-2009-0093137(SRC).
  %The work of F. Huang is partially supported by NSFC Grant No. 11371349,
  %National Basic Research Program of China (973 Program) under Grant
  %No. 2011CB808002, and the CAS Program for Cross $\&$ Cooperative
  %Team of the Science $\&$ Technology Innovation. The work of Y. Wang
  %is partially supported by the National Natural Sciences Foundation
  %of China (No. 11171326 and 11322106)}

\begin{abstract}
It is well-known that the rarefaction wave, one of the basic wave patterns to the hyperbolic conservation laws, is nonlinearly stable to the one-dimensional compressible Navier-Stokes equations (cf. \cite{MN, MN2, LX, NYZ}). In the present paper we proved the time-asymptotically nonlinear stability of the planar rarefaction wave to the two-dimensional compressible and isentropic Navier-Stokes equations, which gives the first stability result of the planar rarefaction wave to the multi-dimensional system with physical viscosities. 
\end{abstract}

\maketitle \centerline{\date}
% \tableofcontents
%%%%%%%%%%%%%%%%%%%%%%%%%%%%%%%%%%%%%%%%%%%%%%%%%%%%%%%%%%%%%%%%%%%
%
%                             Sec 1. Introduction
%
%%%%%%%%%%%%%%%%%%%%%%%%%%%%%%%%%%%%%%%%%%%%%%%%%%%%%%%%%%%%%%%%%%%
\section{Introduction and Main Result}
\setcounter{equation}{0}
We are concerned with the time-asymptotic stability of the planar rarefaction wave to the two-dimensional  compressible, isentropic Navier-Stokes equations which read
\begin{equation}  \label{NS}
\begin{cases}
\displaystyle \rho_t + div(\rho \textbf{u}) = 0,        \qquad \qquad \qquad \qquad \qquad \qquad \qquad \qquad (x, y) \in \bbr \times \bbt, ~t \in\bbr^+, \\
\displaystyle (\rho \textbf{u})_t + div(\rho \textbf{u} \otimes \textbf{u}) + \nabla p(\rho) = \mu\triangle\textbf{u} + (\mu + \lam)\nabla div\textbf{u},
\end{cases}
\end{equation}
where $\rho = \rho(t, x, y) \geq 0, \textbf{u} = \textbf{u}(t, x, y) = (u, v)^{t}(t, x, y)$ and $p=p(t,x,y)$ represent the density, the velocity and the pressure of the fluid, respectively, where and in the sequel the notation $(\cdot,\cdot)^t$ meaning the transpose of the vector $(\cdot,\cdot)$. Here it is assumed that the shear viscosity $\mu$ and the bulk viscosity $\lam$ both are constants satisfying the physical restrictions
\begin{equation} \label{VIS}
  \mu > 0, \qquad \mu + \lam \geq 0,
\end{equation}
and the pressure $p=p(\rho)$ is given by the so-called $\gamma-$law
\[
  p(\rho) = \frac{\rho^\gamma}{\gamma}
\]
with $\gamma\geq 1$ being the fluid constant. For the spatial domain, we consider the case $x\in \bbr$ being the real line and $y\in \bbt:=\bbr/\mathbb{Z}$ being one-dimensional unit flat torus, and therefore the differential operator $div\textbf{f}:=\partial_xf_1+\partial_y f_2$ with $\textbf{f}=(f_1,f_2)^t$, $\nabla:=(\partial_x,\partial_y)^t$ and $\triangle:=\partial^2_x+\partial_y^2$.
The initial values  to the problem \eqref{NS} are imposed by
\begin{equation} \label{NSI}
(\rho,\textbf{u})(0,x,y)=(\rho, u,v)(0, x, y) = (\rho_0, u_0, v_0)(x, y) .
\end{equation}
Since we are concerned with the stability of planar rarefaction wave to the system \eqref{NS}, it is assumed that the far fields conditions on the $x-$direction
\begin{equation}
 (\rho, u, v)(t,x, y)\rightarrow (\rho_\pm, u_\pm, 0), \qquad{\rm as} \quad x\rightarrow \pm\infty,
\end{equation}
with $\rho_\pm>0, u_\pm$ being prescribed constants, and the periodic boundary conditions are imposed on $y\in\mathbb{T}$ for the solution $(\rho,u,v)(t,x,y)$, where the end states $(\rho_\pm, u_\pm)$ are connected by the rarefaction wave solution of the Riemann problem to the corresponding one-dimensional hyperbolic conservation laws
\begin{equation}  \label{ES}
\begin{cases}
\displaystyle \rho_t + (\rho u)_x = 0,               \quad\qquad \qquad x \in \bbr, ~t > 0, \\
\displaystyle (\rho u)_t + (\rho u^2 + p(\rho))_x = 0,
\end{cases}
\end{equation}
with the Riemann initial data
\begin{equation} \label{ini}
(\rho_0^r, u_0^r)(x) = \begin{cases}
(\rho_-, u_-),     \quad x < 0,\\
(\rho_+, u_+),     \quad x > 0.
\end{cases}
\end{equation}

\

It could be expected that the large-time behavior of the solution to the compressible Navier-Stokes equations \eqref{NS}-\eqref{NSI} is closely related to the Riemann problem to the corresponding  two-dimensional compressible Euler equations
\begin{equation}  \label{2ES}
\begin{cases}
\displaystyle \rho_t + (\rho u)_x +(\rho v)_y= 0,               \quad\qquad \qquad (x,y) \in \bbr\times\bbt, ~t > 0, \\
\displaystyle (\rho u)_t + (\rho u^2 + p(\rho))_x +(\rho u v)_y= 0,\\
\di (\rho v)_t + (\rho uv )_x +(\rho v^2+ p(\rho))_y= 0,
\end{cases}
\end{equation}
with the Riemann initial data
\begin{equation} \label{2ini}
(\rho_0, u_0,v_0)(x) = \begin{cases}
(\rho_-, u_-,0),     \quad x < 0,\\
(\rho_+, u_+,0),     \quad x > 0.
\end{cases}
\end{equation}

\

Yet there are essential differences between the one-dimensional Riemann problem \eqref{ES}-\eqref{ini} and the two-dimensional Riemann problem \eqref{2ES}-\eqref{2ini} even with $v$-component continuous on both sides of $x=0$ in \eqref{2ini}. Precisely speaking, it is first proved by Chiodaroli-DeLellis-Kreml \cite{CDK}  and Chiodaroli-Kreml \cite{CK}
that there are infinitely many bounded admissible weak solutions to \eqref{2ES}-\eqref{2ini} satisfying the natural entropy condition for shock Riemann initial data by using the convex integration methods in DeLellis-Szekelyhid \cite{DS} while the construction of weak solution in \cite{CDK, CK} seems essential to the two-dimensional system and can not be applied to one-dimensional problem \eqref{ES}-\eqref{ini}. Then Klingenberg-Markfelder \cite{KM} and Brezina-Chiodaroli-Kreml \cite{BCK} extend the results in \cite{CK, CDK} to the case when the corresponding Riemann initial data contain shock or contact discontinuity. On the other hand, similar to the one-dimensional case, for the Riemann solution only containing rarefaction waves to  \eqref{2ES}-\eqref{2ini}, Chen-Chen\cite{CC} and Feireisl-Kreml \cite{FK},  Feireisl-Kreml-Vasseur \cite{FKV} independently proved the uniqueness of the uniformly bounded admissible weak solution even the rarefaction waves are connected with vacuum states (cf. \cite{CC}).

\

The inviscid Euler system \eqref{2ES} or \eqref{ES} is ideal model for fluids and the real fluids are often described by the corresponding Navier-Stokes equations \eqref{NS} with physical viscosities satisfying the restrictions \eqref{VIS}. As mentioned before, we can expect that the large-time behavior of the solution to the compressible Navier-Stokes equations \eqref{NS}-\eqref{NSI} is determined by the Riemann problem to the corresponding inviscid Euler system \eqref{2ES} or  \eqref{ES}, which contains planar shock wave and rarefaction wave in the genuinely nonlinear characteristic fields and contact discontinuity in the linearly degenerate field.  It is interesting to prove the above expectations in mathematics rigor and to investigate the dynamic stability of the basic wave patterns for the compressible Navier-Stokes equations with physical viscosities, and there are rather satisfactory results in the one-dimensional case. The stability of rarefaction wave to the one-dimensional compressible isentropic Navier-Stokes equations was first proved by Matsumura-Nishihara \cite{MN, MN2} and then was extended to full Navier-Stokes-Fourier system by Liu-Xin \cite{LX} and Nishihara-Yang-Zhao \cite{NYZ}. Then Liu-Yu \cite{LY} proved the stability of rarefaction wave to one-dimensional general $n\times n$ conservation laws system with artificial viscosity by point-wise Green function methods. 

\

However, in the multi-dimensional case, the stability of planar rarefaction wave for the compressible Navier-Stokes equations \eqref{NS} is still open and the existing results are confined to the scalar viscous conservation laws case by Xin \cite{X} and its extended results by Ito \cite{I} and Nishikawa-Nishihara \cite{NN}. Note that Hokari-Matsumura \cite{HM} proved the stability of planar rarefaction wave to some artificial $2\times2$ system with uniformly positive viscosities, while this result crucially depends on the strict positivity of the viscosity matrix and can not be applied to the compressible Navier-Stokes system \eqref{NS} with only partial viscosities. 

\

The main purpose of the present paper is to prove the stability of planar rarefaction wave for the two-dimensional compressible Navier-Stokes equation \eqref{NS} with physical constraints \eqref{VIS}. Compared with the one-dimensional stability results in \cite{MN, MN2}, the additional difficulties here lie in the higher dimensionality, the propagation of rarefaction wave in $y-$direction and its interaction with the wave in $x-$direction. Fortunately, we can overcome these difficulties in the case $y\in \bbt$ and got the stability result by using the elementary energy methods. Here we fully used the physical structures of the system which cause some essential cancellations to close the energy estimates (see details in Lemmas \ref{lemma4.2}-\ref{lemma4.5}).  More precisely, we prove that if the rarefaction wave strength is suitably weak (i.e., $|\rho_+ - \rho_-| + |u_+ - u_-| \ll 1$) and initial values $(\rho_0, u_0, v_0)$ in \eqref{NSI} are suitably close to the planar rarefaction wave, then the initial value problem \eqref{NS}-\eqref{NSI} has a global-in-time smooth solution which tends to the planar rarefaction wave $(\rho^r, u^r, 0)(x/t)$ as $t \to +\infty$. The detailed stability result can be found in Theorem \ref{theorem1} below.

\

Now we fist give the description of the one-dimensional rarefaction wave to \eqref{ES} and the planar rarefaction wave to \eqref{2ES}. It is well-known that the inviscid Euler system \eqref{ES} is a strictly hyperbolic one for $\rho >0$ with two distinct eigenvalues
\[
\lambda_1(\rho, u) = u - \sqrt{p'(\rho)},   \qquad \lambda_2(\rho, u) = u + \sqrt{p'(\rho)}
\]
with the corresponding right eigenvectors denoted by $r_1(\rho,u)$ and $r_2(\rho,u)$, respectively,
and both characteristic fields are genuinely nonlinear. The $i-$Riemann invariant $z_i(\rho, u)~(i=1,2)$ is given by\[
z_i (\rho, u) = u + (-1)^{i+1}\int^{\rho}\frac{\sqrt{p'(s)}}{s}ds,
\]
satisfying $\nabla_{(\rho,u)}z_i(\rho,u)\cdot r_i(\rho,u)\equiv0,~(i=1,2), \forall \rho, u.$
Without loss of generality, here we only consider the 2-rarefaction wave case and the cases of 1-rarefaction wave and the superposition of two rarefaction waves can be proved similarly. It is well-known that if the states $(\rho_\pm, u_\pm)$ satisfy
\begin{equation} \label{RC}
u_+ - \int_{\rho_-}^{\rho_+}\frac{\sqrt{p'(s)}}{s}ds = u_-,  \qquad \lambda_2(\rho_+, u_+)>\lambda_2(\rho_-, u_-),
\end{equation}
i.e., 2-Riemann invariant $z_2(\rho,u)$ is constant and the second eigenvalue $\lambda_2(\rho,u)$ is expanding along the 2-rarefaction wave curve, then 
the Riemann problem \eqref{ES}-\eqref{ini} would admit a self-similar wave fan $(\rho^r, u^r)(x/t)$ which consists of only the constant states and the centered rarefaction waves (cf. \cite{L}). Then the planar rarefaction wave solution to the two-dimensionl compressible Euler equations \eqref{2ES}-\eqref{2ini} is defined by $(\rho^r, u^r,0)(\frac xt)$ with $(\rho^r, u^r)(\frac xt)$ being the one-dimensional rarefaction wave to \eqref{ES}.

\

To be more precise, we assume that
\begin{align}
\begin{aligned} \label{RIC}
&(\rho_0 - \rho_0^r, u_0 - u_0^r, v_0) \in L^2(\bbr \times \bbt), \\
&(\nabla \rho_0, \nabla\textbf{u}_0) \in H^1(\bbr \times \bbt)
\end{aligned}
\end{align}
and set 
\[
\Phi_0^2 = \|(\rho_0 - \rho_0^r, u_0 - u_0^r, v_0)\|^2 + \|(\nabla \rho_0, \nabla\textbf{u}_0)\|_1^2 + |(\rho_+ - \rho_-, u_+ - u_-)|^2,
\]
where $\nabla\textbf{u}_0 = (\nabla u_0, \nabla v_0)^t$. Here $H^k(\bbr \times \bbt)(k \geq 0, k\in \mathbb{Z})$ denotes the usual Sobolev space with the norm $\|\cdot\|_k$. We denote $L^2(\bbr \times \bbt) = H^0(\bbr \times \bbt)$ and simply set $\|\cdot\| = \|\cdot\|_0$.
Then our main result can be stated as follows.

\begin{theorem} \label{theorem1}
	For the planar 2-rarefaction wave $(\rho^r, u^r, 0)(\frac xt)$ connecting the constant states $(\rho_\pm, u_\pm,0)$ satisfying (\ref{RC}) with $\rho_\pm >0$, there exists a positive constant $\epsilon_0$ such that if the initial perturbation around the planar rarefaction wave and the wave strength satisfy $\Phi_0 < \epsilon_0$, then the initial value problem (\ref{NS})-(\ref{NSI}) admits a unique global smooth solution $(\rho, \textbf{u}) = (\rho, u, v)$ satisfying
	\[
	\begin{cases}
	(\rho - \rho^r, u - u^r, v) \in C^0(0, +\infty; L^2(\bbr \times \bbt)), \\
	(\nabla\rho, \nabla\textbf{u}) \in C^0(0, +\infty; H^1(\bbr \times \bbt)), \\
	\nabla^3 \textbf{u} \in L^2(0, +\infty; L^2(\bbr \times \bbt)),
	\end{cases}
	\]
	and the time-asymptotic stability of the planar 2-rarefaction wave holds true in the sense that
	\begin{equation} \label{AB}
	\lim_{t \to +\infty} \sup_{(x, y)\in \bbr \times \bbt} |(\rho, u, v)(t, x, y) - (\rho^r, u^r, 0)(\frac xt)| = 0.
	\end{equation}
	%where $(\rho^r, u^r)$ is the rarefaction wave solution of the Riemann problem (\ref{ES}), (\ref{ini}).
\end{theorem}

\begin{remark}
Theorem \ref{theorem1} gives a first stability result of the planar rarefaction wave to the multi-dimensional system \eqref{NS} with physical viscosities. 
\end{remark}
\begin{remark}
Our stability analysis could also be applied to the time-asymptotic stability of the superposition of 1-rarefaction wave and 2-rarefaction wave to the two-dimensional compressible Navier-Stokes equations \eqref{NS} and the wave interaction estimates as in \cite{MN} will be considered additionally. 
\end{remark}

%The proof of Theorem \ref{theorem1} is based on the stability results of one-dimensional rarefaction wave and an elementary $L^2$ energy method.

In the next section, we first construct a smooth approximate rarefaction wave which tends to the rarefaction wave fan uniformly as the time $t$ tends to infinity. Then we reformulate the system for the perturbation around the approximate rarefaction wave  and then establish the a priori estimates for the perturbation by an elementary $L^2$ energy method. Finally, in the last section, based on these a priori estimates, we proved our main Theorem \ref{theorem1}.

%%%%%%%%%%%%%%%%%%%%%%%%%%%%%%%%%%%%
%
%
%%%%%%%%%%%%%%%%%%%%%%%%%%%%%%%%%%%
\section{Smooth Approximate Rarefaction Wave}
\setcounter{equation}{0}
Since the rarefaction wave is only Lipschitz continuous, we will construct a smooth approximation rarefaction wave  through the Burgers equation as in \cite{MN} and \cite{X}. 
Consider the Riemann problem for the inviscid Burgers equation:
\begin{equation} \label{BE}
\begin{cases}
\displaystyle w_t + ww_x = 0, \\
\displaystyle w(0, x) = w_0^r(x) = \begin{cases}
w_-,  \quad x < 0,\\
w_+,  \quad x > 0.
\end{cases}
\end{cases}
\end{equation}
If $w_- < w_+$, then \eqref{BE} has the rarefaction wave fan $w^r(t, x) = w^r(x/t)$ given by
\begin{equation} \label{BES}
w^r(t, x) = w^r(\frac{x}{t}) = \begin{cases}
w_- , \qquad x < w_-t, \\
\frac{x}{t}, \qquad w_-t \leq x \leq w_+t, \\
w_+ , \qquad x > w_+t.
\end{cases}
\end{equation}
As in \cite{MN}, the approximate rarefaction wave to the compressible Navier-Stokes equations \eqref{NS} can be constructed through the smooth solution of the Burgers equation
\begin{equation} \label{ABE}
\begin{cases}
\displaystyle w_t + ww_x = 0, \\
\displaystyle w(0, x) = w_0(x) = \frac{w_+ + w_-}{2} + \frac{w_+ - w_-}{2} \tanh x.
\end{cases}
\end{equation}
Then, by the characteristic methods, the solution $w(t, x)$ of the problem \eqref{ABE} has the following properties and their proofs can be found in \cite{MN}.
\begin{lemma} \label{lemma2.1}
  Suppose $w_+ > w_-$ and set $\tilde{w} = w_+ - w_-$. Then the problem \eqref{ABE} has a unique smooth global solution $w(t, x)$ such that
   
  (1)~$w_- < w(t, x) < w_+, ~w_x >0$ for $x \in \bbr$ and $t \geq 0$.
   
  (2)~The following estimates hold for all $t > 0$ and $p \in [1, + \infty]$:
  \begin{align*}
    &\|w_x(t, \cdot)\|_{L^p} \leq C \min(|\tilde{w}|, |\tilde{w}|^{1/p}t^{-1+1/p}), \\
    &\|w_{xx}(t, \cdot)\|_{L^p} \leq C \min(|\tilde{w}|,t^{-1}), \\
    &\|w_{xxx}(t, \cdot)\|_{L^p} \leq C \min(|\tilde{w}|,t^{-1}).
  \end{align*}
  
  (3)~$\di\lim_{t \to +\infty} \sup_{x \in \bbr} |w(t, x) - w^r(x/t)| = 0$.
\end{lemma}

We now turn to the approximate rarefaction wave for the Euler system \eqref{ES}-\eqref{ini}. Here and in what follows, the constant states $(\rho_\pm, u_\pm)$ are fixed so that they are connected by the 2-rarefaction wave. Set $w_- = \lambda_2(\rho_-, u_-), w_+ = \lambda_2(\rho_+, u_+)$, and the rarefaction wave strength $\alpha := |\rho_+ - \rho_-| + |u_+ - u_-|$. It is easy to check that the 2-rarefaction wave $(\rho^r, u^r)(t, x) = (\rho^r, u^r)(x/t)$ to the Riemann problem \eqref{ES} - \eqref{RC} is given explicitly by
\begin{align*}
  &\lambda_2(\rho^r, u^r)(t, x) = w^r(t, x), \\
  &z_2(\rho^r, u^r)(t, x) = z_2 (\rho_\pm, u_\pm).
\end{align*}

Correspondingly, the smooth approximate rarefaction wave $(\bar{\rho}, \bar{u})(t, x)$ of the 2-rarefaction wave fan $(\rho^r, u^r)(t, x)$ can be constructed by
\begin{align}
  \begin{aligned} \label{AR}
    &\lambda_2(\bar{\rho}, \bar{u})(t, x) = w(1+t, x), \\
    &z_2(\bar{\rho}, \bar{u})(t, x) = z_2 (\rho_\pm, u_\pm),
  \end{aligned}
\end{align}
where $w(t,x)$ is the smooth solution to the Burgers equation in \eqref{ABE}.
One can easily check that the above approximate rarefaction wave $(\bar{\rho}, \bar{u})$ satisfies the system:
\begin{equation}  \label{ANS}
  \begin{cases}
    \displaystyle \bar{\rho}_t + (\bar{\rho}\bar{u})_x = 0, \\
    \displaystyle (\bar{\rho} \bar{u})_t + (\bar{\rho} \bar{u}^2 + p(\bar{\rho}))_x = 0, \\
    (\bar{\rho}, \bar{u})(0, x) = (\bar{\rho}_0, \bar{u}_0)(x).
  \end{cases}
\end{equation}
The following lemma comes from Lemma \ref{lemma2.1} (cf. \cite{MN}).
\begin{lemma} \label{lemma2.2}
	The smooth approximate 2-rarefaction wave $(\bar{\rho}, \bar{u})$ defined in \eqref{AR} satisfies the following properties:
	
	(1)~$\bar{u}_x = \frac{2}{\gamma + 1}w_x > 0$ for all $x \in \bbr$ and $t \geq 0, \bar{\rho}_x = \bar{\rho}^{\frac{3 - \gamma}{2}} \bar{u}_x$.
	
	(2)~The following estimates hold for all $t \geq 0$ and $p \in [1, + \infty]$:
	\begin{align*}
	&\|(\bar{\rho}_x, \bar{u}_x)\|_{L^p} \leq C \min(\alpha, \alpha^{1/p}(1+t)^{-1+1/p}), \\
	&\|(\bar{\rho}_{xx}, \bar{u}_{xx})\|_{L^p} \leq C \min(\alpha, (1+t)^{-1}), \\
	&\|(\bar{\rho}_{xxx}, \bar{u}_{xxx})\|_{L^p} \leq C \min(\alpha, (1+t)^{-1}).
	\end{align*}
	
	(3)~$\di \lim_{t \to +\infty} \sup_{x \in \bbr} |(\bar{\rho}, \bar{u})(t, x) - (\rho^r, u^r)( \frac xt)| = 0$.
\end{lemma}

%%%%%%%%%%%%%%%%%%%%%%%%%%%%%%%%%%%%
%
%
%%%%%%%%%%%%%%%%%%%%%%%%%%%%%%%%%%%
\section{Reformulation of the problem}
\setcounter{equation}{0}
Now we define the perturbation around the approximate rarefaction wave $(\bar \rho, \bar u, 0)$ in \eqref{ANS} by
\begin{align*}
&\phi(t, x, y) = \rho(t, x, y) - \bar{\rho}(t, x), \\
&\Psi(t, x, y) = (\varphi, \psi)^t(t, x, y) = (u, v)^t(t, x, y) - (\bar{u}, 0)^t(t, x),
\end{align*}
with $(\rho, u, v)(t,x,y)$ being the solution to the problem \eqref{NS}-\eqref{NSI}.
Then we obtain the perturbation system for $(\phi,\Psi):$
\begin{equation}  \label{REF}
    \begin{cases}
    \displaystyle \phi_t + \rho div\Psi + \rho_y\psi + u\phi_x + \bar{\rho}_x\varphi +\bar{u}_x\phi = 0, \\
    \displaystyle \rho \Psi_t + \rho u\Psi_x + \rho \psi \Psi_y + (\rho \bar{u}_x \varphi, 0)^t + p'(\rho) \nabla\phi + ((p'(\rho) -   \frac{\rho}{\bar{\rho}} p'(\bar{\rho})) \bar{\rho}_x, 0)^t \\
         \qquad\qquad = \mu\triangle\Psi + (\mu + \lam)\nabla div\Psi + ((2\mu + \lam)\bar{u}_{xx}, 0)^t,
    \end{cases}
\end{equation}
\begin{equation} \label{REFI}
  (\phi, \Psi)(0, x, y) = (\phi_0, \Psi_0)(x, y) = (\phi_0, \varphi_0, \psi_0)(x, y) = (\rho_0 - \bar{\rho}_0, u_0 - \bar{u}_0, v_0)(x, y) \in H^2,
\end{equation}
and we note that the condition \eqref{RIC} assures \eqref{REFI} holds true. We first choose a positive constant $E_0$ by virtue of Sobolev imbedding theorem such that
\begin{align} \label{SI}
  \sup_{(x, y)\in \bbr \times \bbt} |f(x, y)| \leq \frac{1}{2}\rho_- \qquad \textrm{for any } f \in H^2, ~\| f \|_2 \leq E_0.
\end{align}
Note that if $E_0$ is suitably small, then \eqref{SI} is obviously true.
Then the solution of \eqref{REF}, \eqref{REFI} is sought in the set of functional space  $X(0, +\infty)$, where for $0 \leq T \leq +\infty$, we define
\begin{align*}
  X(0, T) = &\{ (\phi, \Psi)|(\phi, \Psi) \in C^0(0, T; H^2), \nabla \phi \in L^2(0, T; H^1), \nabla \Psi \in L^2(0, T; H^2) \\
	&{\rm and} \sup_{0 \leq t \leq T} \| (\phi, \Psi)(t) \|_2 \leq E_0 \}.
\end{align*}
Then the conditions $\di \sup_{0 \leq t \leq T} \| (\phi, \Psi)(t) \|_2 \leq E_0 $ and \eqref{SI} imply that $| \phi |, | \Psi | \leq \frac{1}{2}\rho_-$ and $|\textbf{u}|=|(u,v)| \leq C$ with $C$ being a positive constant which only depends on $\rho_-, u_\pm$, and therefore, the density function $\rho(t,x,y):=\bar \rho(t,x)+\phi(t,x,y)$ satisfies that 
\begin{equation}\label{db}
0 < \frac{1}{2}\rho_- \leq \rho \leq \frac{1}{2}\rho_- + \rho_+,
\end{equation}
since $0 < \rho_- \leq \bar{\rho} \leq \rho_+.$ Notice that the uniform lower and upper boundness of the density function $\rho(t,x,y)$ in \eqref{db} guarantee the strict parabolicity of the momentum equation $\eqref{NS}_2$, which are crucial for the local and global-in-time existence of the classical solution to the system \eqref{NS}.

\begin{proposition} \label{proposition3.1}
For the planar 2-rarefaction wave $(\rho^r, u^r, 0)(\frac xt)$ connecting the constant states $(\rho_\pm, u_\pm,0)$ satisfying (\ref{RC}), there exist positive constants $\epsilon_0$ and $C$ such that if $\|(\phi_0, \Psi_0)\|_2 + \alpha \leq \epsilon_0$, then the reformulated problem \eqref{REF}-\eqref{REFI} admits a unique global solution $(\phi, \Psi) \in X(0, +\infty)$ satisfying
  \begin{align}
  \begin{aligned} \label{PRO3.1}
    &\sup_{t \geq 0} \| (\phi, \Psi)(t) \|_2^2 + \int_{0}^{+\infty} \Big[\|\bar{u}_x^{1/2} (\phi, \varphi)\|^2 + \|(\nabla\phi, \nabla\Psi)\|_1^2 + \|\nabla^3\Psi\|^2 \Big]dt \\
    &\leq C \big[\|(\phi_0, \Psi_0)\|_2^2 + \alpha^{1/4}\big].
  \end{aligned}
  \end{align}
\end{proposition}
Once we proved Proposition \ref{proposition3.1},  we claim that
\begin{equation} \label{FFT}
  \int_{0}^{+\infty} \| (\nabla \phi, \nabla \Psi) \|^2 + | \frac{d}{dt}\| (\nabla \phi, \nabla \Psi) \|^2 | dt < +\infty.
\end{equation}
In fact, to prove \eqref{FFT} we only need to show $\di \int_{0}^{+\infty} | \frac{d}{dt}\| (\nabla \phi, \nabla \Psi )\|^2 | dt < +\infty.$ By Cauchy's inequality, Lemma \ref{lemma2.2}, \eqref{FFT} and the standard Elliptic estimates, one has
\begin{align*}
  &\int_{0}^{+\infty} | \frac{d}{dt}\| \nabla \phi \|^2 | dt = \int_{0}^{+\infty} | \int_{\bbt}\int_{\bbr}2\nabla\phi \cdot \nabla\phi_t dxdy | dt\\
  &= 2\int_{0}^{+\infty} | \int_{\bbt}\int_{\bbr} div(\phi_t\nabla\phi) - \phi_t\triangle\phi dxdy | dt= 2\int_{0}^{+\infty} | \int_{\bbt}\int_{\bbr} \phi_t\triangle\phi dxdy | dt\\
  &= 2\int_{0}^{+\infty} | \int_{\bbt}\int_{\bbr} (\rho div\Psi + \psi\phi_y + u\phi_x + \bar{\rho}_x\varphi +\bar{u}_x\phi)\triangle\phi dxdy | dt\\
  &\leq C\int_{0}^{+\infty}\int_{\bbt}\int_{\bbr} (\varphi_x^2 + \psi_y^2 + \phi_y^2 + \phi_x^2 + \bar{u}_x\varphi^2 + \bar{u}_x\phi^2 + (\triangle\phi)^2) dxdydt\\
  &\leq C \big[\|(\phi_0, \Psi_0)\|_2^2 + \alpha^{1/4}\big] < +\infty,
\end{align*}
and
\begin{align*}
  &\int_{0}^{+\infty} | \frac{d}{dt}\| \nabla \Psi \|^2 | dt = 2\int_{0}^{+\infty} | \int_{\bbt}\int_{\bbr} \Psi_t \cdot \triangle \Psi dxdy | dt\\
  &\leq C\int_{0}^{+\infty}\int_{\bbt}\int_{\bbr} (|\Psi_x|^2 + |\Psi_y|^2 + \bar{u}_x\varphi^2 + |\nabla\phi|^2 + \bar{u}_x\phi^2 + |\nabla^2\Psi|^2 + \bar{u}_{xx}^2) dxdydt\\
  &\leq C \big[\|(\phi_0, \Psi_0)\|_2^2 + \alpha^{1/4} + \alpha^{2/3}\big]< +\infty.
\end{align*}
Then the inequality \eqref{FFT} implies that
\begin{equation} \label{lim1}
  \lim_{t \to +\infty} \| (\nabla\phi, \nabla\Psi)(t, \cdot) \| = 0.
\end{equation}
By 2-dimensional Sobolev's inequality,% (cf. Xin \cite{X}), 
it holds that
\begin{equation} \label{2SI}
  \sup_{(x, y)\in \bbr \times \bbt} |\phi| \leq C\big[ \| \phi \|^{1/2} \| \phi_x \|^{1/2} + \| \phi_y \|^{1/2} \| \phi_{xy} \|^{1/2}\big].
\end{equation}
By \eqref{lim1}, \eqref{2SI} and Proposition \ref{proposition3.1}, we get
\begin{equation} \label{lim2}
  \lim_{t \to +\infty} \sup_{(x, y)\in \bbr \times \bbt} |(\phi, \Psi)(t, x, y)| = 0.
\end{equation}
Thus combining \eqref{lim2} with Lemma \ref{lemma2.2} (3), we have the desired time-asymptotic behavior in \eqref{AB}. Noting that
\begin{align*}
  &\|(\phi_0, \Psi_0)\| \leq \|(\rho_0 - \rho_0^r, u_0 - u_0^r, v_0)\| + \|(\bar{\rho}_0 - \rho_0^r, \bar u_0 - \bar{u}_0^r)\| \leq \|(\rho_0 - \rho_0^r, u_0 - u_0^r, v_0)\| + C \alpha,\\
  &\|(\nabla\phi_0, \nabla \Psi_0) \|_1 \leq \|(\nabla \rho_0, \nabla \textbf{u}_0)\|_1 + \|(\bar{\rho}_{0x}, \bar{u}_{0x})\| + \|(\bar{\rho}_{0xx}, \bar{u}_{0xx})\| \leq \|(\nabla \rho_0, \nabla \textbf{u}_0)\|_1 + C \alpha,
\end{align*}
for some positive constant $C$, thus all statements in Theorem \ref{theorem1} follow.

Since the proof for the local-in-time existence and uniqueness of the classical solution to \eqref{REF}-\eqref{REFI} is standard (for instance, one can refer to \cite{So}), in particular for the suitably small perturbation of the solution around the planar rarefaction wave satisfying the property \eqref{db}, the details will be omitted. To prove Proposition \ref{proposition3.1}, it suffices to show the following a priori estimates.
\begin{proposition} (a priori estimate) \label{proposition3.2}
Suppose that the reformulated problem \eqref{REF}-\eqref{REFI} has a solution $(\phi, \Psi) \in X(0, T)$ for some $T(>0)$. Then there exist positive constants $\epsilon_1$ and $C$ which are independent of  $T$ such that if
  \[
    \sup_{0 \leq t \leq T} \| (\phi, \Psi)(t) \|_2 + \alpha \leq \epsilon_1,
  \]
  then it holds 
  \begin{align}
	\begin{aligned} \label{PRO3.2}
	  &\sup_{0 \leq t \leq T} \| (\phi, \Psi)(t) \|_2^2 + \int_{0}^{T} \Big[\|\bar{u}_x^{1/2} (\phi,  \varphi)\|^2 + \|(\nabla\phi, \nabla\Psi)\|_1^2 + \|\nabla^3\Psi\|^2\Big] dt \\
	  &\leq C \big[\|(\phi_0, \Psi_0)\|_2^2 + \alpha^{1/4}\big].
	\end{aligned}
  \end{align}
\end{proposition}

%%%%%%%%%%%%%%%%%%%%%%%%%%%%%%%%%%%%
%
%
%%%%%%%%%%%%%%%%%%%%%%%%%%%%%%%%%%%
\section{A Priori Estimates}
\setcounter{equation}{0}
In this section, we will prove Proposition \ref{proposition3.2}.
Throughout this section we suppose that $(\rho_\pm, u_\pm)$ satisfying \eqref{RC} with $\rho_\pm >0, u_\pm \in \bbr$ being fixed and that the problem \eqref{REF}-\eqref{REFI} has a solution $(\phi, \Psi) \in X(0, T)$ for some $T(>0)$. For simplicity, we write $C$ as generic positive constants which may depend on $(\rho_\pm, u_\pm)$ but are independent of $T$, and set $E = \di\sup_{0 \leq t \leq T} \| (\phi, \Psi)(t) \|_2$.

\begin{lemma} \label{lemma4.1}
  There exists a positive constant $C$ such that for $0 \leq t \leq T$,
  \begin{align}
    \begin{aligned} \label{LEM4.1}
    \sup_{0 \leq t \leq T} \| (\phi, \Psi)(t) \|^2 + \int_{0}^{T}\big[ \|\bar{u}_x^{1/2} (\phi,  \varphi) \|^2 + \|\nabla \Psi\|^2\big] dt
    \leq C \big[\| (\phi_0, \Psi_0) \|^2 + \alpha^{1/4}\big].
    \end{aligned}
  \end{align}
\end{lemma}
\begin{proof}
  First, multiplying the second equation of \eqref{REF} by $\Psi$ gives
  \begin{align}
    \begin{aligned} \label{01}
      &(\frac{1}{2}\rho|\Psi|^2)_t + \frac{1}{2}div(\rho\textbf{u}|\Psi|^2) - \mu div(\varphi\nabla\varphi + \psi\nabla\psi) - (\mu + \lam) div(\Psi div\Psi) \\
      &\quad + \rho\bar{u}_x\varphi^2 + \mu|\nabla\Psi|^2 + (\mu + \lam) (div\Psi)^2 + p'(\rho)\nabla\phi\cdot\Psi + (p'(\rho) - \frac{\rho}{\bar{\rho}} p'(\bar{\rho})) \bar{\rho}_x\varphi \\
      &= (2\mu + \lam)\bar{u}_{xx}\varphi.
    \end{aligned}
  \end{align}
Define the potential energy by
  \[
    \Phi(\rho , \bar{\rho}) = \int_{\bar{\rho}}^{\rho} \frac{p(s) - p(\bar{\rho})}{s^2} ds
 = \frac{1}{(\gamma-1)\rho}(p(\rho) - p(\bar{\rho}) - p'(\bar{\rho})\phi).
\]
Direct computations yield that
  \begin{align}
    \begin{aligned} \label{02}
      &(\rho\Phi)_t + div(\rho\textbf{u}\Phi + (p(\rho) - p(\bar{\rho}))\Psi) + \bar{u}_x(p(\rho) - p(\bar{\rho}) - p'(\bar{\rho})\phi) - p'(\rho)\nabla\phi\cdot\Psi \\
      &\quad - (p'(\rho) - \frac{\rho}{\bar{\rho}} p'(\bar{\rho})) \bar{\rho}_x\varphi = 0.
    \end{aligned}
  \end{align}
Combining \eqref{01} and \eqref{02} together and then integrating the resulted equation over $[0, ~t]\times\bbr\times\bbt$ imply
  \begin{align}
	\begin{aligned} \label{03}
	  \|(\phi, \Psi)\|^2(t) + \int_{0}^{t} \big[\|\bar{u}_x^{1/2} (\phi, \varphi) \|^2 + \|\nabla \Psi\|^2\big] dt
	  \leq C \| (\phi_0, \Psi_0) \|^2 + C \int_{0}^{t}\int_{\bbt}\int_{\bbr}|\bar{u}_{xx}\varphi| dxdydt.
	\end{aligned}
  \end{align}
  By the one-dimensional Sobolev's inequality, H\"{o}lder's inequality, Young's inequality and Lemma \ref{lemma2.2}, one has
  \begin{align}
  \begin{aligned} \label{04}
    & \int_{0}^{t}\int_{\bbt}\int_{\bbr}|\bar{u}_{xx}\varphi|dxdydt
    \leq  \int_{0}^{t}\int_{\bbt}\|\bar{u}_{xx}\|_{L_x^1}\|\varphi\|_{L_x^\infty}dydt \\
    &\leq C \int_{0}^{t}\int_{\bbt}\|\bar{u}_{xx}\|_{L_x^1}\|\varphi\|_{L_x^2}^{1/2}\|\varphi_x\|_{L_x^2}^{1/2}dydt \\
    &\leq \frac{1}{4}\int_{0}^{t}\int_{\bbt}\|\varphi_x\|_{L_x^2}^2dydt + C \int_{0}^{t}\int_{\bbt}\|\bar{u}_{xx}\|_{L_x^1}^{4/3}\|\varphi\|_{L_x^2}^{2/3}dydt \\
    &\leq \frac{1}{4}\int_{0}^{t}\|\varphi_x\|^2dt + C \int_{0}^{t}(\int_{\bbt}\|\bar{u}_{xx}\|_{L_x^1}^2dy)^{2/3}(\int_{\bbt}\|\varphi\|_{L_x^2}^2dy)^{1/3}dt \\
    &\leq \frac{1}{4}\int_{0}^{t}\|\varphi_x\|^2dt + C \int_{0}^{t} \|\bar{u}_{xx}\|_{L_x^1}^{4/3} \|\varphi\|^{2/3} dt \\
    &\leq \frac{1}{4}\int_{0}^{t}\|\varphi_x\|^2dt + C \sup_{0 \leq t \leq T} \|\varphi\|^{2/3} \int_{0}^{t} \|\bar{u}_{xx}\|_{L_x^1}^{4/3} dt \\
    &\leq \frac{1}{4}\int_{0}^{t}\|\varphi_x\|^2dt + C \sup_{0 \leq t \leq T} \|\varphi\|^{2/3} \int_{0}^{t} (\min(\alpha, (1+t)^{-1}))^{4/3} dt \\
    &\leq \frac{1}{4}\int_{0}^{t}\|\varphi_x\|^2dt + C \sup_{0 \leq t \leq T} \|\varphi\|^{2/3} \int_{0}^{t} (\alpha^{1/8}(1+t)^{-7/8})^{4/3} dt \\
   % &\leq \frac{1}{4}\int_{0}^{t}\|\varphi_x\|^2dt + C \sup_{0 \leq t \leq T} \|\varphi\|^{2/3} \int_{0}^{t} \alpha^{1/6}(1+t)^{-7/6} dt \\
    &\leq \frac{1}{4}\int_{0}^{t}\|\varphi_x\|^2dt + C \alpha^{1/6} \sup_{0 \leq t \leq T} \|\varphi\|^{2/3} \leq \frac{1}{4}\int_{0}^{t}\|\varphi_x\|^2dt + \frac{1}{4}\sup_{0 \leq t \leq T} \|\varphi\|^2 + C \alpha^{1/4}. \\
  \end{aligned}
  \end{align}
Thus \eqref{03} and \eqref{04} yield \eqref{LEM4.1} in Lemma \ref{lemma4.1}.
\end{proof}

\begin{lemma} \label{lemma4.2}
  There exists a positive constant $C$ such that for $0 \leq t \leq T$,
  \begin{align}
	\begin{aligned} \label{LEM4.2}
 	  &\|(\phi, \Psi)\|^2(t) + \|\nabla\phi\|^2 (t)+ \int_{0}^{t} \big[\|\bar{u}_x^{1/2} (\phi, \varphi)\|^2 + \|(\nabla\phi, \nabla\Psi)\|^2\big] dt\\
	  &\leq C \big[\|(\phi_0, \Psi_0)\|^2+ \|\nabla\phi_0\|^2 + \alpha^{1/4}\big] + C(E + \alpha) \int_{0}^{t} \big[\|\bar{u}_x^{1/2}(\phi,\varphi)\|^2 + \|(\nabla\phi, \nabla\Psi)\|^2\big] dt.
	\end{aligned}
  \end{align}
\end{lemma}
\begin{proof}
 Multiplying the second equation of \eqref{REF} by $\frac{\nabla\phi}{\rho}$ gives
  \begin{align}
    \begin{aligned} \label{12}
      &(\Psi \cdot \nabla\phi)_t - div(\Psi\phi_t) -(u\phi\psi_y)_x + (u\phi\psi_x)_y + \frac{p'(\rho)}{\rho}|\nabla\phi|^2 \\
      &\quad - \frac{(\mu + \lam)\nabla div\Psi\cdot\nabla\phi+\mu\triangle\Psi\cdot\nabla\phi}{\rho} \\
      &= \rho(div\Psi)^2 + \psi(\phi_y\varphi_x - \phi_x\varphi_y) + \phi(\varphi_y\psi_x - \varphi_x\psi_y) - \bar{u}_x(\phi\varphi_x - \varphi\phi_x) \\
      &\quad + \bar{\rho}_x \varphi div\Psi- (\frac{p'(\rho)}{\rho} - \frac{p'(\bar{\rho})}{\bar{\rho}}) \bar{\rho}_x\phi_x + \frac{2\mu + \lam}{\rho}\bar{u}_{xx} \phi_x.
    \end{aligned}
  \end{align}
  Applying the operator $\nabla$ to the first equation of \eqref{REF} and then multiplying the resulted equation by $\nabla\phi/\rho^2$ yield
  \begin{align}
	\begin{aligned} \label{11}
	  &(\frac{|\nabla\phi|^2}{2\rho^2})_t + div(\frac{\textbf{u} |\nabla\phi|^2}{2\rho^2}) + \frac{\bar{u}_x \phi_x^2}{\rho^2} +\frac{\nabla{\rm div} \Psi\cdot \nabla\phi}{\rho}\\
	  &= - \frac{\phi_x \nabla\phi\cdot\nabla\varphi}{\rho^2} - \frac{\phi_y \nabla\phi\cdot\nabla\psi}{\rho^2} + \frac{|\nabla\phi|^2 div\Psi}{2\rho^2} - \frac{(\bar{\rho}_{xx}\varphi + \bar{u}_{xx}\phi) \phi_x}{\rho^2} - \frac{\bar{\rho}_x \nabla\phi\cdot\nabla\varphi}{\rho^2} \\
	  &\quad - \frac{\bar{\rho}_x\phi_x div\Psi}{\rho^2}  + \frac{\bar{u}_x|\nabla\phi|^2}{2\rho^2},
	\end{aligned}
  \end{align}
 By the fact that
  $$
  \begin{array}{ll}
 \di  \frac{\nabla{\rm div} \Psi\cdot \nabla\phi}{\rho}=\frac{(\mu + \lam)\nabla\phi\cdot\nabla div\Psi + \mu\nabla\phi\cdot\triangle\Psi}{(2\mu + \lam)\rho} \\
  \di + \frac{\mu}{2\mu + \lam}((\frac{\phi_x\psi_x}{\rho})_y - (\frac{\phi_y\psi_x}{\rho})_x + (\frac{\phi_y\varphi_y}{\rho})_x - (\frac{\phi_x\varphi_y}{\rho})_y)+\frac{\mu\bar{\rho}_x\phi_y (\varphi_y - \psi_x)}{(2\mu + \lam)\rho^2},
  \end{array}
  $$
we observe that both last terms on the left hand side of \eqref{11} and \eqref{12} will be cancelled if we multiply the equality \eqref{11} by $2\mu + \lam$ and then add the two resulted equalities together, which is one of the main observations of the present paper and quite important to close a priori estimates. Otherwise, the first order derivative estimate will depend on the second order one, and so on, then how to close a priori estimates is the main issue. Here we fully used the physical structures of the system for the above cancellations to close the a priori estimates.
 
Now multiplying the equality \eqref{11} by $2\mu + \lam$ ,and then adding the resulted equation and \eqref{12} together, and then integrating the final equation over $[0, ~t]\times\bbr\times\bbt$, we have
  \begin{align}
    \begin{aligned} \label{13}
      &\Big[\int_{\bbt}\int_{\bbr} \big(\frac{2\mu + \lam}{2\rho^2}|\nabla\phi|^2 + \Psi\cdot\nabla\phi \big)dxdy\Big]\Big|_0^t + \int_{0}^{t}\int_{\bbt}\int_{\bbr}\Big[ \frac{2\mu + \lam}{\rho^2}\bar{u}_x \phi_x^2 + \frac{p'(\rho)}{\rho}|\nabla\phi|^2 \Big]dxdydt \\
      &= \int_{0}^{t}\int_{\bbt}\int_{\bbr} \Big[\rho (div\Psi)^2 - \frac{2\mu + \lam}{\rho^2}\phi_x \nabla\phi \cdot \nabla\varphi - \frac{2\mu+ \lam}{\rho^2}\phi_y \nabla\phi \cdot \nabla\psi + \frac{2\mu +\lam}{2\rho^2}|\nabla\phi|^2 div\Psi \\
      &\quad - \frac{2\mu + \lam}{\rho^2}(\bar{\rho}_{xx}\varphi + \bar{u}_{xx}\phi) \phi_x - \frac{2\mu + \lam}{\rho^2}\bar{\rho}_x \nabla\phi\cdot\nabla\varphi - \frac{2\mu + \lam}{\rho^2}\bar{\rho}_x\phi_x div\Psi - \frac{\mu}{\rho^2}\bar{\rho}_x\phi_y (\varphi_y - \psi_x) \\
      &\quad + \frac{2\mu + \lam}{2\rho^2}\bar{u}_x|\nabla\phi|^2 + \psi(\phi_y\varphi_x - \phi_x\varphi_y) + \phi(\varphi_y\psi_x - \varphi_x\psi_y) - \bar{u}_x(\phi\varphi_x - \varphi\phi_x) + \bar{\rho}_x \varphi div\Psi \\
      &\quad - (\frac{p'(\rho)}{\rho} - \frac{p'(\bar{\rho})}{\bar{\rho}}) \bar{\rho}_x\phi_x + \frac{2\mu + \lam}{\rho}\bar{u}_{xx} \phi_x \Big]dxdydt.
    \end{aligned}
  \end{align}
  The combination of \eqref{LEM4.1} and \eqref{13} leads to
  \begin{align}
    \begin{aligned} \label{14}
      &\| (\phi, \Psi)\|^2(t) + \|\nabla\phi\|^2(t) + \int_{0}^{t} \big[\|\bar{u}_x^{1/2} (\phi, \varphi,\phi_x) \|^2 + \|\nabla(\phi, \Psi)\|^2\big] dt\\
      &\leq C\big[\| (\phi_0, \Psi_0) \|^2 + \|\nabla\phi_0\|^2 + \alpha^{1/4}\big] + C \int_{0}^{t}\int_{\bbt}\int_{\bbr} \Big[|\frac{2\mu + \lam}{\rho^2}\phi_x \nabla\phi \cdot \nabla\varphi| + |\frac{2\mu + \lam}{\rho^2}\bar{\rho}_{xx}\varphi\phi_x| \\
      &\quad + |\frac{2\mu + \lam}{\rho^2}\bar{\rho}_x \nabla\phi\cdot\nabla\varphi| + |\frac{2\mu + \lam}{2\rho^2}\bar{u}_x|\nabla\phi|^2| + |\psi\phi_y\varphi_x| + |\phi\varphi_y\psi_x| + |\bar{u}_x\phi\varphi_x| \\
      &\quad + |(\frac{p'(\rho)}{\rho} - \frac{p'(\bar{\rho})}{\bar{\rho}}) \bar{\rho}_x\phi_x| + |\frac{2\mu + \lam}{\rho}\bar{u}_{xx} \phi_x|\Big]dxdydt + C \int_{0}^{t}\int_{\bbt}\int_{\bbr} \Big[|\frac{2\mu+ \lam}{\rho^2}\phi_y \nabla\phi \cdot \nabla\psi| \\
      &\quad + |\frac{2\mu +\lam}{2\rho^2}|\nabla\phi|^2 div\Psi| + |\frac{2\mu + \lam}{\rho^2}\bar{u}_{xx}\phi\phi_x| + |\frac{2\mu + \lam}{\rho^2}\bar{\rho}_x\phi_x div\Psi| + |\frac{\mu}{\rho^2}\bar{\rho}_x\phi_y (\varphi_y - \psi_x)| \\
      &\quad + |\psi\phi_x\varphi_y| + |\phi\varphi_x\psi_y| + |\bar{u}_x\varphi\phi_x| + |\bar{\rho}_x \varphi div\Psi|\Big] dxdydt \\
      &:= C \big[\| \phi_0, \Psi_0 \|^2 + \|\nabla\phi_0\|^2 + \alpha^{1/4}\big] + \sum_{i=1}^9 G_i + H.
    \end{aligned}
  \end{align}
  We estimate each term on the right-hand side of \eqref{14}. By Young's inequality and interpolation inequality, it holds that
  \begin{align*}
  G_1 &= C \int_{0}^{t} \| \frac{2\mu + \lam}{\rho^2}\phi_x \nabla\phi \cdot \nabla\varphi \|_{L^1} dt 
       \leq \frac{1}{80} \int_{0}^{t} \|\phi_x\|^2 dt + C \int_{0}^{t} (\|\nabla\phi\|_{L^4}^4 + \|\nabla\varphi\|_{L^4}^4) dt \\
      &\leq \frac{1}{80} \int_{0}^{t} \|\phi_x\|^2 dt + C \int_{0}^{t} (\|\nabla\phi\|^2 \|\nabla\phi\|_1^2 + \|\nabla\varphi\|^2 \|\nabla\varphi\|_1^2) dt \\
      &\leq \frac{1}{80} \int_{0}^{t} \|\phi_x\|^2 dt + CE \int_{0}^{t} (\|\nabla\phi\|^2 + \|\nabla\varphi\|^2) dt.
  \end{align*}
  By Young's inequality, interpolation inequality and Lemma \ref{lemma2.2}, one has
  \begin{align*}
  G_2 &= C \int_{0}^{t} \|\frac{2\mu + \lam}{\rho^2}\bar{\rho}_{xx}\varphi\phi_x\|_{L^1} dt 
       \leq \frac{1}{80} \int_{0}^{t} \|\phi_x\|^2 dt +  C \int_{0}^{t}\|\bar{\rho}_{xx}\|_{L^\infty}^2 \|\varphi\|^2 dt \\
      &\leq \frac{1}{80} \int_{0}^{t} \|\phi_x\|^2 dt + C \int_{0}^{t} \alpha^{1/2}(1+t)^{-3/2} \|\varphi\|^2  dt \\
      &\leq \frac{1}{80} \int_{0}^{t} \|\phi_x\|^2 dt + C\alpha \sup_{0\leq t\leq T} \|\psi\|^2.
  \end{align*}
  It follows from Cauchy's inequality and Lemma \ref{lemma2.2} that
  \begin{align*}
  G_3 &= C \int_{0}^{t} \|\frac{2\mu + \lam}{\rho^2}\bar{\rho}_x \nabla\phi\cdot\nabla\varphi\|_{L^1} dt
       \leq \frac{1}{80} \int_{0}^{t} \|\nabla\varphi\|^2 dt + C \int_{0}^{t} \|\bar{\rho}_x \nabla\phi\|^2 dt \\
      &\leq \frac{1}{80} \int_{0}^{t} \|\nabla\varphi\|^2 dt + C \alpha \int_{0}^{t} \|\nabla\phi\|^2 dt.
  \end{align*}
  It follows from Lemma \ref{lemma2.2} that
  \begin{align*}
  G_4 = C \int_{0}^{t} \|\frac{2\mu + \lam}{2\rho^2}\bar{u}_x|\nabla\phi|^2\|_{L^1} dt      
      \leq C \alpha \int_{0}^{t} \|\nabla\phi\|^2 dt.
  \end{align*}
 Similar to estimate $G_1$, we can estimate $G_5$ and $G_6$ as 
  \begin{align*}
  G_5 = C \int_{0}^{t} \| \psi\phi_y\varphi_x \|_{L^1} dt 
     \leq \frac{1}{80} \int_{0}^{t} \|\phi_y\|^2 dt + CE \int_{0}^{t} (\|\nabla\psi\|^2 + \|\varphi_x\|^2) dt,
  \end{align*}
  and
  \begin{align*}
  G_6 = C \int_{0}^{t} \| \phi\varphi_y\psi_x \|_{L^1} dt 
     \leq \frac{1}{80} \int_{0}^{t} \|\varphi_y\|^2 dt + CE \int_{0}^{t} (\|\nabla\phi\|^2 + \|\psi_x\|^2) dt.
  \end{align*}
  It follows from Cauchy's inequality and Lemma \ref{lemma2.2} that
  \begin{align*}
  G_7 &= C \int_{0}^{t} \|\bar{u}_x\phi\varphi_x\|_{L^1} dt
      \leq \frac{1}{80} \int_{0}^{t} \|\varphi_x\|^2 dt + C \int_{0}^{t} \|\bar{u}_x \phi\|^2 dt \\
      &\leq \frac{1}{80} \int_{0}^{t} \|\varphi_x\|^2 dt + C \alpha \int_{0}^{t} \|\bar{u}_x^{1/2} \phi\|^2 dt,
  \end{align*}
  \begin{align*}
  G_8 &= C \int_{0}^{t} \|(\frac{p'(\rho)}{\rho} - \frac{p'(\bar{\rho})}{\bar{\rho}}) \bar{\rho}_x\phi_x\|_{L^1} dt
      \leq \frac{1}{80} \int_{0}^{t} \|\phi_x\|^2 dt + C \int_{0}^{t} \|\bar{\rho}_x\|^2 dt \\
      &\leq \frac{1}{80} \int_{0}^{t} \|\phi_x\|^2 dt + C \int_{0}^{t} (\alpha^{1/3}(1 + t)^{-2/3})^2 dt 
      \leq \frac{1}{80} \int_{0}^{t} \|\phi_x\|^2 dt + C \alpha^{2/3},
  \end{align*}
  and
  \begin{align*}
  G_9 &= C \int_{0}^{t} \|\frac{2\mu + \lam}{\rho}\bar{u}_{xx} \phi_x\|_{L^1} dt
      \leq \frac{1}{80} \int_{0}^{t} \|\phi_x\|^2 dt + C \int_{0}^{t} \|\bar{u}_{xx}\|^2 dt \\
      &\leq \frac{1}{80} \int_{0}^{t} \|\phi_x\|^2 dt + C \alpha^{2/3}.
  \end{align*}
Then all  the terms in $H$ can be analyzed similarly and the details will be omitted for brevity.
Substituting the estimates of $G_1-G_9$ and $H$ into \eqref{14}, we can prove \eqref{LEM4.2} in Lemma \ref{lemma4.2}.
\end{proof}

\begin{lemma} \label{lemma4.3}
  There exists a positive constant $C$ such that for $0 \leq t \leq T$,
  \begin{align}
	\begin{aligned} \label{LEM4.3}
	  &\|(\phi, \Psi)\|_1^2(t) + \int_{0}^{t} \big[\|\bar{u}_x^{1/2}(\phi, \varphi)\|^2 + \|(\nabla\phi, \nabla\Psi)\|^2 + \|\nabla^2\Psi\|^2\big] dt \\
	  &\leq C \big[\|(\phi_0, \Psi_0)\|_1^2 + \alpha^{1/4}\big] + C(E + \alpha) \int_{0}^{t} \big[\|\bar{u}_x^{1/2}(\phi, \varphi)\|^2 + \|(\nabla\phi, \nabla\Psi)\|^2\big] dt.
	\end{aligned}
  \end{align}
\end{lemma}
\begin{proof}
  Multiplying the second equation of \eqref{REF} by $-\triangle\Psi/\rho$ gives
  \begin{align}
    \begin{aligned} \label{15}
      &(\frac{|\nabla\Psi|^2}{2})_t - div(\varphi_t\nabla\varphi + \psi_t\nabla\psi + \frac{\mu + \lam}{\rho}div\Psi\nabla div\Psi - \frac{\mu + \lam}{\rho}div\Psi\triangle\Psi) \\
      &\quad + (\frac{u}{2}|\Psi_y|^2 - \frac{u}{2}|\Psi_x|^2)_x - (u\Psi_x\cdot\Psi_y)_y + \frac{1}{2}\bar{u}_x|\Psi_x|^2 + \frac{\mu}{\rho}|\triangle\Psi|^2 + \frac{\mu + \lam}{\rho}|\nabla div\Psi|^2 \\
      &= - \varphi_y\psi_x\psi_y - \frac{1}{2}\varphi_x|\Psi_x|^2 + \frac{1}{2}\varphi_x(\psi_y^2 - \varphi_y^2) + \psi\Psi_y\cdot\triangle\Psi + \frac{1}{2} \bar{u}_x |\Psi_y|^2 + \bar{u}_x\varphi\triangle\varphi \\
      &\quad + \frac{p'(\rho)}{\rho} \nabla\phi\cdot\triangle\Psi + (\frac{p'(\rho)}{\rho} - \frac{p'(\bar{\rho})}{\bar{\rho}}) \bar{\rho}_x \triangle\varphi
      + \frac{\mu + \lam}{\rho^2}div\Psi\nabla\phi\cdot\nabla div\Psi \\
      &\quad - \frac{\mu + \lam}{\rho^2}div\Psi\nabla\phi\cdot\triangle\Psi + \frac{\mu + \lam}{\rho^2}\bar{\rho}_x div\Psi div\Psi_x - \frac{\mu + \lam}{\rho^2}\bar{\rho}_x div\Psi\triangle\varphi - \frac{2\mu + \lam}{\rho}\bar{u}_{xx} \triangle\varphi.
    \end{aligned}
  \end{align}
  Integrating the above equation over $[0, ~t]\times\bbr\times\bbt$ yields
  \begin{align}
    \begin{aligned} \label{16}
      &\|\nabla\Psi\|^2(t) + \int_{0}^{t} \big[\|\bar{u}_x^{1/2}\Psi_x\|^2 + \|\triangle\Psi\|^2\big] dt \\
      &\leq C \|\nabla\Psi_0\|^2 + C \int_{0}^{t}\int_{\bbt}\int_{\bbr} \Big[|\varphi_y\psi_x\psi_y| + |\psi\Psi_y\cdot\triangle\Psi| + |\frac{1}{2} \bar{u}_x |\Psi_y|^2| + |\bar{u}_x\varphi\triangle\varphi| \\
      &\quad + |\frac{p'(\rho)}{\rho} \nabla\phi\cdot\triangle\Psi| + |(\frac{p'(\rho)}{\rho} - \frac{p'(\bar{\rho})}{\bar{\rho}}) \bar{\rho}_x \triangle\varphi| + |\frac{\mu + \lam}{\rho^2}div\Psi\nabla\phi\cdot\nabla div\Psi| \\
      &\quad + |\frac{\mu + \lam}{\rho^2}\bar{\rho}_x div\Psi div\Psi_x| + |\frac{2\mu + \lam}{\rho}\bar{u}_{xx} \triangle\varphi|\Big] dxdydt + C \int_{0}^{t}\int_{\bbt}\int_{\bbr} \Big[|\frac{1}{2}\varphi_x|\Psi_x|^2| \\
      &\quad + |\frac{1}{2}\varphi_x(\psi_y^2 - \varphi_y^2)| + |\frac{\mu + \lam}{\rho^2}div\Psi\nabla\phi\cdot\triangle\Psi| + |\frac{\mu + \lam}{\rho^2}\bar{\rho}_x div\Psi\triangle\varphi|\Big]dxdydt \\
      &:= C \|\nabla\Psi_0\|^2 + \sum_{i = 1}^9 I_i + J.
    \end{aligned}
  \end{align}
  We estimate each term on the right-hand side of \eqref{16} as follows. By Young's inequality and interpolation inequality, it holds that
  \begin{align*}
  I_1 &= C \int_{0}^{t} \| \varphi_y\psi_x\psi_y \|_{L^1} dt
       \leq \frac{1}{80} \int_{0}^{t} \|\varphi_y\|^2 dt + C \int_{0}^{t} (\|\psi_x\|_{L^4}^4 + \|\psi_y\|_{L^4}^4) dt \\
      &\leq \frac{1}{80} \int_{0}^{t} \|\varphi_y\|^2 dt + C \int_{0}^{t} (\|\psi_x\|_1^4 + \|\psi_y\|_1^4) dt\\
      &\leq \frac{1}{80} \int_{0}^{t} \|\varphi_y\|^2 dt + CE \int_{0}^{t} \|\nabla\psi\|_1^2 dt.
  \end{align*}
  Similarly, we have
  \begin{align*}
  I_2 = C \int_{0}^{t} \| \psi\Psi_y\cdot\triangle\Psi \|_{L^1} dt
      \leq \frac{1}{80} \int_{0}^{t} \|\triangle\Psi\|^2 dt + CE \int_{0}^{t} (\|\nabla\psi\|^2 + \|\Psi_y\|^2) dt.
  \end{align*}
  By Lemma \ref{lemma2.2}, one has
  \begin{align*}
  I_3 = C \int_{0}^{t} \| \frac{1}{2} \bar{u}_x |\Psi_y|^2 \|_{L^1} dt
      \leq C \alpha \int_{0}^{t} \|\Psi_y\|^2 dt.
  \end{align*}
  It follows from Cauchy's inequality and Lemma \ref{lemma2.2} that
  \begin{align*}
  I_4 = C \int_{0}^{t} \| \bar{u}_x\varphi\triangle\varphi \|_{L^1} dt
      \leq \frac{1}{80} \int_{0}^{t} \|\triangle\varphi\|^2 dt + C \alpha \int_{0}^{t} \|\bar{u}_x^{1/2} \varphi\|^2 dt,
  \end{align*}
  \begin{align*}
  I_5 = C \int_{0}^{t} \| \frac{p'(\rho)}{\rho} \nabla\phi\cdot\triangle\Psi \|_{L^1} dt
      \leq \frac{1}{80} \int_{0}^{t} \|\triangle\Psi\|^2 dt + C \int_{0}^{t} \|\nabla\phi\|^2 dt
  \end{align*}
  and
  \begin{align*}
  I_6 &= C \int_{0}^{t} \| (\frac{p'(\rho)}{\rho} - \frac{p'(\bar{\rho})}{\bar{\rho}}) \bar{\rho}_x \triangle\varphi \|_{L^1} dt
       \leq \frac{1}{80} \int_{0}^{t} \|\triangle\varphi\|^2 dt + C \int_{0}^{t} \| \bar{\rho}_x \phi\|^2 dt \\
      &\leq \frac{1}{80} \int_{0}^{t} \|\triangle\varphi\|^2 dt + C \alpha\int_0^t\|\bar u_x^{1/2}\phi\|^2dt .
  \end{align*}
  By Young's inequality and interpolation inequality, one has
  \begin{align*}
  I_7 &= C \int_{0}^{t} \| \frac{\mu + \lam}{\rho^2}div\Psi\nabla\phi\cdot\nabla div\Psi \|_{L^1} dt \\
      &\leq \frac{1}{80} \int_{0}^{t} (\|\nabla\varphi_x\|^2 + \|\nabla\psi_y\|^2) dt + C \int_{0}^{t} (\|div\Psi\|_{L^4}^4 + \|\nabla\phi\|_{L^4}^4) dt \\
      &\leq \frac{1}{80} \int_{0}^{t} (\|\nabla\varphi_x\|^2 + \|\nabla\psi_y\|^2) dt + C \int_{0}^{t} (\|div\Psi\|^2\|div\Psi\|_1^2 + \|\nabla\phi\|^2\|\nabla\phi\|_1^2) dt\\
      &\leq \frac{1}{80} \int_{0}^{t} (\|\nabla\varphi_x\|^2 + \|\nabla\psi_y\|^2) dt + CE \int_{0}^{t} (\|\varphi_x\|^2 + \|\psi_y\|^2 + \|\nabla\phi\|^2) dt.
  \end{align*}
  It follows from Cauchy's inequality and Lemma \ref{lemma2.2} that
  \begin{align*}
  I_8 &= C \int_{0}^{t} \| \frac{\mu + \lam}{\rho^2}\bar{\rho}_x div\Psi div\Psi_x \|_{L^1} dt
     \leq \frac{1}{80} \int_{0}^{t} (\|\varphi_{xx}\|^2 + \|\psi_{xy}\|^2) dt + C \int_{0}^{t} \|\bar{\rho}_x div\Psi\|^2 dt \\
     &\leq \frac{1}{80} \int_{0}^{t} (\|\varphi_{xx}\|^2 + \|\psi_{xy}\|^2) dt + C \alpha \int_{0}^{t} (\|\varphi_x\|^2 + \|\psi_y\|^2) dt,
  \end{align*}
  and
  \begin{align*}
  I_9 &= C \int_{0}^{t} \| \frac{2\mu + \lam}{\rho}\bar{u}_{xx} \triangle\varphi \|_{L^1} dt
       \leq \frac{1}{80} \int_{0}^{t} \|\triangle\varphi\|^2 dt + C \int_{0}^{t} \|\bar{u}_{xx}\|^2 dt \\
      &\leq \frac{1}{80} \int_{0}^{t} \|\triangle\varphi\|^2 dt + C \alpha^{2/3}.
  \end{align*}
All the terms in $J$ can be analyzed similarly as in $I_i~(i=1,2,\cdots 9)$.  Then substituting the resulted estimates for $I_1-I_9$ and $J$ into \eqref{16} and the Elliptic estimate $\|\triangle\Psi\| \sim \|\nabla^2\Psi\|$ give
  \begin{align}
    \begin{aligned} \label{17}
      &\|\nabla\Psi\|^2 + \int_{0}^{t} \|\nabla^2\Psi\|^2 dt\leq C \big(\|\nabla\Psi_0\|^2 + \alpha^{2/3}\big)\\
      & \qquad + C (E + \alpha) \int_{0}^{t} \big[\|\bar{u}_x^{1/2} \varphi\|^2 + \|(\nabla \phi, \nabla \Psi)\|^2\big] dt + C \int_{0}^{t} \|(\nabla\phi, \nabla\varphi)\|^2 dt .
    \end{aligned}
  \end{align}
  Combining \eqref{LEM4.2} and \eqref{17}, we complete the proof of Lemma \ref{lemma4.3}.
\end{proof}

\begin{lemma} \label{lemma4.4}
  There exists a positive constant $C$ such that for $0 \leq t \leq T$,
  \begin{align}
	\begin{aligned} \label{LEM4.4}
	  &\|(\phi, \Psi)\|_1^2(t)  + \|\nabla^2\phi\|^2(t) + \int_{0}^{t}\big[ \|\bar{u}_x^{1/2}(\phi, \varphi)\|^2 + \|(\nabla\phi, \nabla\Psi)\|_1^2 \big]dt \\
	  &\leq C \big[\|(\phi_0, \Psi_0)\|_1^2+ \|\nabla^2\phi_0\|^2 + \alpha^{1/4}\big]+ CE \int_{0}^{t} \|\nabla^3\Psi\|^2 dt \\
	  &\quad + C(E + \alpha) \int_{0}^{t} \big[\|\bar{u}_x^{1/2}(\phi, \varphi)\|^2 + \|(\nabla\phi, \nabla\Psi)\|_1^2\big]dt .
	\end{aligned}
  \end{align}
\end{lemma}
\begin{proof}
  Applying  the operator $\nabla^2$  on the first equation of \eqref{REF} and then multiplying the resulted equation by $\nabla^2\phi/\rho^2$, we have
  \begin{align}
    \begin{aligned} \label{21}
	  &(\frac{|\nabla^2\phi|^2}{2\rho^2})_t + div(\frac{\textbf{u} |\nabla^2\phi|^2}{2\rho^2}) + \frac{\mu}{2\mu + \lam}((\frac{\nabla\phi_x\cdot\nabla\psi_x}{\rho})_y - (\frac{\nabla\phi_y\cdot\nabla\psi_x}{\rho})_x + (\frac{\nabla\phi_y\cdot\nabla\varphi_y}{\rho})_x \\
	  &\quad - (\frac{\nabla\phi_x\cdot\nabla\varphi_y}{\rho})_y) + \frac{2\bar{u}_x |\nabla\phi_x|^2}{\rho^2} + \frac{(\mu + \lam)\nabla^2\phi\cdot\nabla^2 div\Psi + \mu\nabla^2\phi\cdot\nabla\triangle\Psi}{(2\mu + \lam)\rho} \\
	  &= K(t, x, y) + L(t, x, y),
	\end{aligned}
  \end{align}
  where 
  \begin{align*}
    &K(t, x, y) = - \frac{\phi_x \nabla\phi_x\cdot\nabla div\Psi}{\rho^2} - \frac{\phi_{xy} \nabla\psi\cdot\nabla\phi_x}{\rho^2} + \frac{\bar{u}_x |\nabla^2\phi|^2}{2\rho^2} + \frac{\mu\bar{\rho}_x\nabla\phi_y\cdot\nabla\psi_x}{(2\mu + \lam)\rho^2} \\
    &\quad - \frac{\bar{\rho}_{xx}div\Psi\phi_{xx}}{\rho^2} - \frac{\bar{u}_{xx}\phi_x\phi_{xx}}{\rho^2} - \frac{\bar{\rho}_{xxx}\varphi\phi_{xx}}{\rho^2} - \frac{\bar{u}_{xxx}\phi\phi_{xx}}{\rho^2},
  \end{align*}
  and
  \begin{align*}
    &L(t, x, y) = - \frac{\phi_y \nabla\phi_y\cdot\nabla div\Psi}{\rho^2} - \frac{\nabla\phi\cdot(div\Psi_x\nabla\phi_x + div\Psi_y\nabla\phi_y)}{\rho^2} - \frac{\phi_{yy} \nabla\psi\cdot\nabla\phi_y}{\rho^2} \\
    &\quad - \frac{(\psi_x\nabla\phi_x + \psi_y\nabla\phi_y)\cdot\nabla\phi_y}{\rho^2} - \frac{(\phi_x\nabla^2\varphi + \phi_y\nabla^2\psi)\cdot\nabla^2\phi}{\rho^2} - \frac{(\varphi_x\nabla\phi_x + \varphi_y\nabla\phi_y)\cdot\nabla\phi_x}{\rho^2} \\
    &\quad - \frac{\nabla\varphi\cdot(\phi_{xx}\nabla\phi_x + \phi_{xy}\nabla\phi_y)}{\rho^2} + \frac{|\nabla^2\phi|^2 div\Psi}{2\rho^2} - \frac{\mu\phi_y\nabla\phi_x\cdot\nabla\psi_x}{(2\mu + \lam)\rho^2} + \frac{\mu\phi_x\nabla\phi_y\cdot\nabla\psi_x}{(2\mu + \lam)\rho^2} \\
    &\quad - \frac{\mu\phi_x\nabla\phi_y\cdot\nabla\varphi_y}{(2\mu + \lam)\rho^2} + \frac{\mu\phi_y\nabla\phi_x\cdot\nabla\varphi_y}{(2\mu + \lam)\rho^2} - \frac{\mu\bar{\rho}_x\nabla\phi_y\cdot\nabla\varphi_y}{(2\mu + \lam)\rho^2} - \frac{2\bar{\rho}_x\nabla\phi_x\cdot\nabla div\Psi}{\rho^2} - \frac{\bar{\rho}_x \nabla^2\phi\cdot\nabla^2\varphi}{\rho^2} \\
    &\quad - \frac{2\bar{\rho}_{xx} \nabla\varphi\cdot\nabla\phi_x}{\rho^2} - \frac{2\bar{u}_{xx} \nabla\phi\cdot\nabla\phi_x}{\rho^2}.   
  \end{align*}
Then dividing the second equation of \eqref{REF} by $\rho$, applying the operator $\nabla$ on the resulted equation and then multiplying the final equation by $\nabla^2\phi$, we have
  \begin{align}
    \begin{aligned} \label{22}
     &(\nabla\Psi \cdot \nabla^2\phi)_t - div(\phi_{xt}\nabla\varphi + \phi_{yt}\nabla\psi) - (u\nabla\phi\cdot\Psi_{yy})_x + (u\nabla\phi\cdot\Psi_{xy})_y + \frac{p'(\rho)}{\rho}|\nabla^2\phi|^2 \\
     &\quad - \frac{\mu\nabla^2\phi\cdot\nabla\triangle\Psi +(\mu + \lam)\nabla^2\phi\cdot\nabla^2 div\Psi}{\rho}= \rho\nabla div\Psi\cdot\triangle\Psi + M(t, x, y) + N(t, x, y),
    \end{aligned}
  \end{align}
  where 
  \begin{align*}
    &M(t, x, y) = div\Psi\nabla\phi\cdot\triangle\Psi + \psi\nabla\phi_y\cdot\triangle\Psi - \varphi_x\nabla\varphi\cdot\nabla\phi_x - \frac{p'''(\rho)}{\gamma - 1}\phi_x\nabla\phi\cdot\nabla\phi_x \\
    &\quad - \frac{\mu}{\rho^2}\triangle\varphi\nabla\phi\cdot\nabla\phi_x - \frac{\mu + \lam}{\rho^2}div\Psi_x\nabla\phi\cdot\nabla\phi_x + \bar{\rho}_x div\Psi\triangle\varphi + \bar{u}_x\nabla\phi\cdot\Psi_{xx} + \bar{\rho}_{xx} \varphi\triangle\varphi \\
    &\quad + \bar{u}_{xx} \phi\triangle\varphi - (\frac{p'''(\rho)}{\gamma - 1} - \frac{p'''(\bar{\rho})}{\gamma - 1}) \bar{\rho}_x^2\phi_{xx} - (\frac{p'(\rho)}{\rho} - \frac{p'(\bar{\rho})}{\bar{\rho}}) \bar{\rho}_{xx}\phi_{xx} -  \frac{\mu}{\rho^2}\bar{\rho}_x\nabla\phi_x\cdot\triangle\Psi \\
    &\quad - \frac{2\mu + \lam}{\rho^2}\bar{u}_{xx}\nabla\phi\cdot\nabla\phi_x - \frac{2\mu + \lam}{\rho^2}\bar{\rho}_x\bar{u}_{xx}\phi_{xx} + \frac{2\mu + \lam}{\rho}\bar{u}_{xxx}\phi_{xx},
  \end{align*}
  and
  \begin{align*}
    &N(t, x, y) = \phi_y\nabla\psi\cdot\triangle\Psi + \phi_x\nabla\varphi\cdot\triangle\Psi + \varphi_y\nabla\phi\cdot\Psi_{xy} - \varphi_x\nabla\phi\cdot\Psi_{yy} - \psi_x\nabla\varphi\cdot\nabla\phi_y \\
    &\quad - \psi\nabla\Psi_y\cdot\nabla^2\phi - \varphi_y\nabla\psi\cdot\nabla\phi_x - \psi_y\nabla\psi\cdot\nabla\phi_y - \frac{p'''(\rho)}{\gamma - 1}\phi_y\nabla\phi\cdot\nabla\phi_y - \frac{\mu}{\rho^2}\triangle\psi\nabla\phi\cdot\nabla\phi_y \\
    &\quad - \frac{\mu + \lam}{\rho^2}div\Psi_y\nabla\phi\cdot\nabla\phi_y + \bar{u}_x\phi_x\triangle\varphi + \bar{\rho}_x\nabla\varphi\cdot\triangle\Psi - \bar{u}_x\Psi_x\cdot\nabla\phi_x - \bar{u}_{xx}\varphi\phi_{xx} \\
    &\quad - \bar{u}_x\nabla\varphi\cdot\nabla\phi_x - \frac{2p'''(\rho)}{\gamma - 1}\bar{\rho}_x\nabla\phi\cdot\nabla\phi_x - \frac{\mu + \lam}{\rho^2}\bar{\rho}_x\nabla\phi_x\cdot\nabla div\Psi.
  \end{align*}
Now multiplying \eqref{21} by $2\mu + \lam$, adding the resulted equation and \eqref{22} together, using the same cancellations as in Lemma \ref{lemma4.2} on both last terms in the left hand side of \eqref{21} and \eqref{22}, and then integrating the final equation over $[0, ~t]\times\bbr\times\bbt$, one has
   \begin{equation} \label{23}
      \begin{array}{l}
      \di \Big[\int_{\bbt}\int_{\bbr} \Big(\frac{2\mu + \lam}{2\rho^2}|\nabla^2\phi|^2 + \nabla\Psi\cdot\nabla^2\phi \Big)dxdy\Big]\Big|_0^t \\[3mm]
      \di + \int_{0}^{t}\int_{\bbt}\int_{\bbr} \Big[\frac{2(2\mu + \lam)}{\rho^2}\bar{u}_x |\nabla\phi_x|^2 + \frac{p'(\rho)}{\rho}|\nabla^2\phi|^2 \Big]dxdydt \\[3mm]
      \di = \int_{0}^{t}\int_{\bbt}\int_{\bbr} \Big[\rho\nabla div\Psi\cdot\triangle\Psi + (2\mu + \lam)K(t, x, y) + (2\mu + \lam)L(t, x, y) + M(t, x, y) \\
      \di \qquad \qquad\qquad + N(t, x, y) \Big]dxdydt.
      \end{array}
    \end{equation}
  The combination of \eqref{LEM4.3} and \eqref{23} leads to
%  \begin{align}
    \begin{equation} \label{24}
      \begin{array}{l}
     \di \|(\phi, \Psi)\|_1^2 (t) + \|\nabla^2\phi\|^2(t) + \int_{0}^{t} \big[\|\bar{u}_x^{1/2}(\phi,\varphi,\nabla\phi_x)\|^2+ \|(\nabla\phi, \nabla\Psi)\|_1^2 \big]dt\\
      \di \leq C\big[\|(\phi_0, \Psi_0)\|_1^2+ \|\nabla^2\phi_0\|^2 + \alpha^{1/4})\big]+ C(E + \alpha) \int_{0}^{t}\big[\|\bar{u}_x^{1/2}(\phi, \varphi)\|^2+ \|(\nabla\phi, \nabla\Psi)\|^2\big] dt  \\
      \di \quad + C\int_{0}^{t}\int_{\bbt}\int_{\bbr}\big[|K(t, x, y)| + |L(t, x, y)| + |M(t, x, y)| + |N(t, x, y)|\big]dxdydt.
      \end{array}
    \end{equation}
%  \end{align}
  Now we estimate the terms on the right-hand side of \eqref{24}. First, we consider the terms in $K(t, x, y)$ as 
  \begin{align}
    \begin{aligned} \label{25}
      &C \int_{0}^{t}\int_{\bbt}\int_{\bbr} |K(t, x, y)| dxdydt \\
      &\leq C\int_{0}^{t}\int_{\bbt}\int_{\bbr} \Big[|\frac{\phi_x \nabla\phi_x\cdot\nabla div\Psi}{\rho^2}| + |\frac{\phi_{xy} \nabla\psi\cdot\nabla\phi_x}{\rho^2}| + |\frac{\bar{u}_x |\nabla^2\phi|^2}{2\rho^2}| + |\frac{\mu\bar{\rho}_x\nabla\phi_y\cdot\nabla\psi_x}{(2\mu + \lam)\rho^2}| \\
      &\quad + |\frac{\bar{\rho}_{xx}div\Psi\phi_{xx}}{\rho^2}| + |\frac{\bar{u}_{xx}\phi_x\phi_{xx}}{\rho^2}| + |\frac{\bar{\rho}_{xxx}\varphi\phi_{xx}}{\rho^2}| + |\frac{\bar{u}_{xxx}\phi\phi_{xx}}{\rho^2}| \Big]dxdydt =\sum_{i=1}^8 K_i.
    \end{aligned}
  \end{align}
  Each term in \eqref{25} will be estimated as follows. By Young's inequality and interpolation inequality, one has
  \begin{align*}
    K_1 &= C \int_{0}^{t} \|\frac{\phi_x \nabla\phi_x\cdot\nabla div\Psi}{\rho^2}\|_{L^1} dt
        \leq \frac{1}{80} \int_{0}^{t} \|\nabla\phi_x\|^2 dt + C \int_{0}^{t} (\|\phi_x\|_{L^4}^4 + \|\nabla div\Psi\|_{L^4}^4) dt \\
        &\leq \frac{1}{80} \int_{0}^{t} \|\nabla\phi_x\|^2 dt + C \int_{0}^{t} (\|\phi_x\|^2 \|\phi_x\|_1^2 + \|\nabla div\Psi\|^2 \|\nabla div\Psi\|_1^2) dt \\
        &\leq \frac{1}{80} \int_{0}^{t} \|\nabla\phi_x\|^2 dt + CE \int_{0}^{t} (\|(\phi_x,\nabla\phi_x)\|^2 + \|(\nabla div \Psi, \nabla^2 div\Psi)\|^2) dt.
  \end{align*}
  By H\"{o}lder's inequality, Cauchy's inequality and Sobolev's inequality, it holds that
  \begin{align*}
    K_2 &= C \int_{0}^{t} \|\frac{\phi_{xy} \nabla\psi\cdot\nabla\phi_x}{\rho^2}\|_{L^1} dt
         \leq C \int_{0}^{t} \|\nabla\psi\|_{L^\infty} \|\phi_{xy}\| \|\nabla\phi_x\| dt \\
        &\leq CE \int_{0}^{t} \|\nabla\psi\|_2 \|\phi_{xy}\| dt
         \leq CE \int_{0}^{t} (\|\nabla\psi\| + \|\nabla^2\psi\| + \|\nabla^3\psi\|) \|\phi_{xy}\| dt \\
        &\leq CE \int_{0}^{t} (\|\phi_{xy}\|^2 + \|\nabla\psi\|^2 + \|\nabla^2\psi\|^2 + \|\nabla^3\psi\|^2) dt.
  \end{align*}
  By Lemma \ref{lemma2.2}, one has
  \begin{align*}
    K_3 = C \int_{0}^{t} \| \frac{\bar{u}_x |\nabla^2\phi|^2}{2\rho^2} \|_{L^1} dt
       \leq C \alpha \int_{0}^{t} \|\nabla^2\phi\|^2 dt.
  \end{align*}
  It follows from Cauchy's inequality and Lemma \ref{lemma2.2} that
  \begin{align*}
    K_4 = C \int_{0}^{t} \|\frac{\mu\bar{\rho}_x\nabla\phi_y\cdot\nabla\psi_x}{(2\mu + \lam)\rho^2}\|_{L^1} dt 
       \leq \frac{1}{80} \int_{0}^{t} \|\nabla\phi_y\|^2 dt + C \alpha \int_{0}^{t} \|\nabla\psi_x\|^2 dt.
  \end{align*}
  By Young's inequality, interpolation inequality and Lemma \ref{lemma2.2}, one has
  \begin{align*}
    K_5 &= C \int_{0}^{t} \|\frac{\bar{\rho}_{xx}div\Psi\phi_{xx}}{\rho^2}\|_{L^1} dt
        \leq \frac{1}{80} \int_{0}^{t} \|\phi_{xx}\|^2 dt + C \int_{0}^{t} \|\bar{\rho}_{xx}\|_{L^4}^4 dt  + C \int_{0}^{t} \|div\Psi\|_{L^4}^4 dt \\
        &\leq \frac{1}{80} \int_{0}^{t} \|\phi_{xx}\|^2 dt + C \int_{0}^{t} \|div\Psi\|^2 \| div\Psi\|_1^2 dt + C \int_{0}^{t} (\alpha^{1/2}(1 + t)^{-1/2})^4 dt\\          
        &\leq \frac{1}{80} \int_{0}^{t} \|\phi_{xx}\|^2 dt + CE \int_{0}^{t} \|(div \Psi,\nabla div\Psi)\|^2 dt + C \alpha^2.
  \end{align*}
Similarly, one can obtain
  \begin{align*}
    K_6 = C \int_{0}^{t} \|\frac{\bar{u}_{xx}\phi_x\phi_{xx}}{\rho^2}\|_{L^1} dt
       \leq \frac{1}{80} \int_{0}^{t} \|\phi_{xx}\|^2 dt + CE \int_{0}^{t} \|\nabla\phi_x\|^2 dt + C \alpha^2,
  \end{align*}
  \begin{align*}
    K_7 = C \int_{0}^{t} \|\frac{\bar{\rho}_{xxx}\varphi\phi_{xx}}{\rho^2}\|_{L^1} dt
       \leq \frac{1}{80} \int_{0}^{t} \|\phi_{xx}\|^2 dt + CE \int_{0}^{t} \|\nabla\varphi\|^2 dt + C \alpha^2,
  \end{align*}
  and
  \begin{align*}
    K_8 = C \int_{0}^{t} \|\frac{\bar{u}_{xxx}\phi\phi_{xx}}{\rho^2}\|_{L^1} dt
       \leq \frac{1}{80} \int_{0}^{t} \|\phi_{xx}\|^2 dt + CE \int_{0}^{t} \|\nabla\phi\|^2 dt + C \alpha^2.
  \end{align*}
The terms in $L(t, x, y)$ can be handled similarly and the details will be omitted for brevity. In the following, we will estimate $M(t, x, y)$ and $N(t, x, y)$ and we only estimate $M(t, x, y)$ for brevity.
First, it holds that
  \begin{align}
    \begin{aligned} \label{26}
      &C \int_{0}^{t}\int_{\bbt}\int_{\bbr} |M(t, x, y)| dxdydt \\
      &\leq C\int_{0}^{t}\int_{\bbt}\int_{\bbr}\Big[ |div\Psi\nabla\phi\cdot\triangle\Psi| + |\psi\nabla\phi_y\cdot\triangle\Psi| + |\varphi_x\nabla\varphi\cdot\nabla\phi_x| + |\frac{p'''(\rho)}{\gamma - 1}\phi_x\nabla\phi\cdot\nabla\phi_x| \\
      &\quad + |\frac{\mu}{\rho^2}\triangle\varphi\nabla\phi\cdot\nabla\phi_x| + |\frac{\mu + \lam}{\rho^2}div\Psi_x\nabla\phi\cdot\nabla\phi_x| + |\bar{\rho}_x div\Psi\triangle\varphi| + |\bar{u}_x\nabla\phi\cdot\Psi_{xx}| + |\bar{\rho}_{xx} \varphi\triangle\varphi| \\
      &\quad + |\bar{u}_{xx} \phi\triangle\varphi| + |(\frac{p'''(\rho)}{\gamma - 1} - \frac{p'''(\bar{\rho})}{\gamma - 1}) \bar{\rho}_x^2\phi_{xx}| + |(\frac{p'(\rho)}{\rho} - \frac{p'(\bar{\rho})}{\bar{\rho}}) \bar{\rho}_{xx}\phi_{xx}| + |\frac{\mu}{\rho^2}\bar{\rho}_x\nabla\phi_x\cdot\triangle\Psi| \\
      &\quad + |\frac{2\mu + \lam}{\rho^2}\bar{u}_{xx}\nabla\phi\cdot\nabla\phi_x| + |\frac{2\mu + \lam}{\rho^2}\bar{\rho}_x\bar{u}_{xx}\phi_{xx}| + |\frac{2\mu + \lam}{\rho}\bar{u}_{xxx}\phi_{xx}| \Big]dxdydt
      = \sum_{i = 1}^{16} M_i.
    \end{aligned}
  \end{align}
  Next we estimate the right hand side of \eqref{26} terms by terms.
  By Young's inequality and interpolation inequality, it holds that
  \begin{align*}
    M_1 &= C \int_{0}^{t} \|div\Psi\nabla\phi\cdot\triangle\Psi\|_{L^1} dt
        \leq \frac{1}{80} \int_{0}^{t} \|\triangle\Psi\|^2 dt + C \int_{0}^{t} (\|div\Psi\|_{L^4}^4 + \|\nabla\phi\|_{L^4}^4) dt \\
        &\leq \frac{1}{80} \int_{0}^{t} \|\triangle\Psi\|^2 dt + C \int_{0}^{t} (\|div\Psi\|^2 \| div\Psi\|_1^2 + \|\nabla\phi\|^2 \|\nabla\phi\|_1^2) dt \\
        &\leq \frac{1}{80} \int_{0}^{t} \|\triangle\Psi\|^2 dt + CE \int_{0}^{t} (\|(div\Psi,\nabla div\Psi)\|^2 + \|(\nabla\phi,\nabla^2\phi)\|^2) dt,
  \end{align*}
  and
  \begin{align*}
    M_2 &= C \int_{0}^{t} \|\psi\nabla\phi_y\cdot\triangle\Psi\|_{L^1} dt        \leq \frac{1}{80} \int_{0}^{t} \|\nabla\phi_y\|^2 dt + C E\int_{0}^{t} \|\triangle\Psi\|^2 dt.
  \end{align*}
 Similarly, one has
  \begin{align*}
    M_3 &= C \int_{0}^{t} \|\varphi_x\nabla\varphi\cdot\nabla\phi_x\|_{L^1} dt\\
    &
       \leq \frac{1}{80} \int_{0}^{t} \|\nabla\phi_x\|^2 dt + CE \int_{0}^{t} (\|(\varphi_x,\nabla\varphi_x)\|^2 + \|(\nabla\varphi,\nabla^2\varphi)\|^2) dt,
  \end{align*}
  \begin{align*}
    M_4 &= C \int_{0}^{t} \|\frac{p'''(\rho)}{\gamma - 1}\phi_x\nabla\phi\cdot\nabla\phi_x\|_{L^1} dt\\
    &\leq \frac{1}{80} \int_{0}^{t} \|\nabla\phi_x\|^2 dt + CE \int_{0}^{t} (\|(\phi_x,\nabla\phi_x)\|^2 + \|(\nabla\phi,\nabla^2\phi)\|^2) dt,
  \end{align*}
  \begin{align*}
    M_5 &= C \int_{0}^{t} \|\frac{\mu}{\rho^2}\triangle\varphi\nabla\phi\cdot\nabla\phi_x\|_{L^1} dt\\
    &
       \leq \frac{1}{80} \int_{0}^{t} \|\nabla\phi_x\|^2 dt + CE \int_{0}^{t} (\|(\triangle\varphi,\nabla\triangle\varphi)\|^2 + \|(\nabla\phi,\nabla^2\phi)\|^2) dt,
  \end{align*}
  and
  \begin{align*}
    M_6 &= C \int_{0}^{t} \|\frac{\mu + \lam}{\rho^2}div\Psi_x\nabla\phi\cdot\nabla\phi_x\|_{L^1} dt \\
       &\leq \frac{1}{80} \int_{0}^{t} \|\nabla\phi_x\|^2 dt + CE \int_{0}^{t} (\|(div\Psi_x,\nabla div\Psi_x)\|^2 + \|(\nabla\phi,\nabla^2\phi)\|^2) dt.
  \end{align*}
  It follows from Cauchy's inequality and Lemma \ref{lemma2.2} that
  \begin{align*}
    M_7 = C \int_{0}^{t} \|\bar{\rho}_x div\Psi\triangle\varphi\|_{L^1} dt 
       \leq \frac{1}{80} \int_{0}^{t} \|\triangle\varphi\|^2 dt + C \alpha \int_{0}^{t} (\|\varphi_x\|^2 + \|\psi_y\|^2) dt,
  \end{align*}
and
  \begin{align*}
    M_8 = C \int_{0}^{t} \|\bar{u}_x\nabla\phi\cdot\Psi_{xx}\|_{L^1} dt 
       \leq \frac{1}{80} \int_{0}^{t} \|\Psi_{xx}\|^2 dt + C \alpha \int_{0}^{t} \|\nabla\phi\|^2 dt.
  \end{align*}
  By Young's inequality, interpolation inequality and Lemma \ref{lemma2.2}, one has
  \begin{align*}
    M_9 &= C \int_{0}^{t} \|\bar{\rho}_{xx} \varphi\triangle\varphi\|_{L^1} dt
        \leq \frac{1}{80} \int_{0}^{t} \|\triangle\varphi\|^2 dt + C \int_{0}^{t} \|\bar{\rho}_{xx}\|_{L^\infty}^2 \|\varphi\|^2 dt \\
        &\leq \frac{1}{80} \int_{0}^{t} \|\triangle\varphi\|^2 dt + C\alpha^{1/2} \sup_{0\leq t\leq T} \| \varphi\|^2,
  \end{align*}
  and
  \begin{align*}
    M_{10} = C \int_{0}^{t} \|\bar{u}_{xx} \phi\triangle\varphi\|_{L^1} dt
          \leq \frac{1}{80} \int_{0}^{t} \|\triangle\varphi\|^2 dt + CE \int_{0}^{t} \|\nabla \phi\|^2 dt + C \alpha^2.
  \end{align*}
Then it follows from Cauchy's inequality and Lemma \ref{lemma2.2} that
  \begin{align*}
    M_{11} &= C \int_{0}^{t} \|(\frac{p'''(\rho)}{\gamma - 1} - \frac{p'''(\bar{\rho})}{\gamma - 1}) \bar{\rho}_x^2\phi_{xx}\|_{L^1} dt
           \leq \frac{1}{80} \int_{0}^{t} \|\phi_{xx}\|^2 dt + C \int_{0}^{t} \|\bar{\rho}_x\|_{L^4}^4 dt \\
           &\leq \frac{1}{80} \int_{0}^{t} \|\phi_{xx}\|^2 dt + C \alpha^2,
  \end{align*}
  \begin{align*}
    M_{12} &= C \int_{0}^{t} \|(\frac{p'(\rho)}{\rho} - \frac{p'(\bar{\rho})}{\bar{\rho}}) \bar{\rho}_{xx}\phi_{xx}\|_{L^1} dt
           \leq \frac{1}{80} \int_{0}^{t} \|\phi_{xx}\|^2 dt + C \int_{0}^{t} \|\bar{\rho}_{xx}\|^2 dt \\
           &\leq \frac{1}{80} \int_{0}^{t} \|\phi_{xx}\|^2 dt + C \alpha^{2/3},
  \end{align*}
  and
  \begin{align*}
    M_{13} = C \int_{0}^{t} \|\frac{\mu}{\rho^2}\bar{\rho}_x\nabla\phi_x\cdot\triangle\Psi\|_{L^1} dt
          \leq \frac{1}{80} \int_{0}^{t} \|\triangle\Psi\|^2 dt + C \alpha \int_{0}^{t} \|\nabla\phi_x\|^2 dt.
  \end{align*}
Similar to $M_9$, we have
  \begin{align*}
    M_{14} = C \int_{0}^{t} \|\frac{2\mu + \lam}{\rho^2}\bar{u}_{xx}\nabla\phi\cdot\nabla\phi_x\|_{L^1} dt
          \leq \frac{1}{80} \int_{0}^{t} \|\nabla\phi_x\|^2 dt + CE \int_{0}^{t} \|\nabla^2 \phi\|^2 dt + C \alpha^2.
  \end{align*}
  It follows from Cauchy's inequality and Lemma \ref{lemma2.2} that
  \begin{align*}
    M_{15} &= C \int_{0}^{t} \|\frac{2\mu + \lam}{\rho^2}\bar{\rho}_x\bar{u}_{xx}\phi_{xx}\|_{L^1} dt
          \leq C \int_{0}^{t} \|\bar{\rho}_x\phi_{xx}\|^2 dt + C \int_{0}^{t} \|\bar{u}_{xx}\|^2 dt \\
           &\leq C \alpha \int_{0}^{t} \|\phi_{xx}\|^2 dt + C \alpha^{2/3},
  \end{align*}
  and
  \begin{align*}
    M_{16} &= C \int_{0}^{t} \|\frac{2\mu + \lam}{\rho}\bar{u}_{xxx}\phi_{xx}\|_{L^1} dt
          \leq \frac{1}{80} \int_{0}^{t} \|\phi_{xx}\|^2 dt + C \int_{0}^{t} \|\bar{u}_{xxx}\|^2 dt \\
           &\leq \frac{1}{80} \int_{0}^{t} \|\phi_{xx}\|^2 dt + C \alpha^{2/3}.
  \end{align*}
  Substituting the estimates of $K(t, x, y), L(t, x, y), M(t, x, y)$ and $N(t, x, y)$ into \eqref{24} and using the Elliptic estimates $\|\triangle\Psi\| \sim\|\nabla^2\Psi\|$ and $\|\nabla\triangle\Psi\| \sim \|\nabla^3\Psi\|$, we can complete the proof of Lemma \ref{lemma4.4}.
\end{proof}

\begin{lemma} \label{lemma4.5}
  There exists a positive constant $C$ such that for $0 \leq t \leq T$,
  \begin{align}
	\begin{aligned} \label{LEM4.5}
	  &\|(\phi, \Psi)\|_2^2(t) + \int_{0}^{t}\Big[ \|\bar{u}_x^{1/2}(\phi, \varphi)\|^2 + \|(\nabla\phi, \nabla\Psi)\|_1^2 + \|\nabla^3\Psi\|^2 \Big]dt \\
	  &\leq C \big[\|(\phi_0, \Psi_0)\|_2^2 + \alpha^{1/4}\big] + C(E + \alpha) \int_{0}^{t} \Big[\|\bar{u}_x^{1/2}(\phi, \varphi)\|^2 + \|(\nabla\phi, \nabla\Psi)\|_1^2 + \|\nabla^3\Psi\|^2\Big] dt.
	\end{aligned}
  \end{align}
\end{lemma}
\begin{proof}
  We first divide the second equation of \eqref{REF} by $\rho$, then applying the operator $\nabla$ to the resulted equation and then multiply the final equation by $-\nabla\triangle\Psi$ to obtain
  \begin{align}
    \begin{aligned} \label{27}
      &(\frac{|\nabla^2\Psi|^2}{2})_t - div(\varphi_{tx}\nabla\varphi_x + \varphi_{ty}\nabla\varphi_y + \psi_{tx}\nabla\psi_x + \psi_{ty}\nabla\psi_y + \frac{\mu + \lam}{\rho}div\Psi_x\nabla div\Psi_x \\
      & + \frac{\mu + \lam}{\rho}div\Psi_y\nabla div\Psi_y - \frac{\mu + \lam}{\rho}(div\Psi_x\triangle\Psi_x + div\Psi_y\triangle\Psi_y)) + (\frac{u}{2}|\Psi_{yy}|^2 - \frac{u}{2}|\Psi_{xx}|^2)_x \\
      &- (u\Psi_{xx}\cdot\Psi_{xy} + u\Psi_{xy}\cdot\Psi_{yy})_y + \frac{1}{2}\bar{u}_x|\Psi_{xx}|^2 + \frac{\mu}{\rho}|\nabla\triangle\Psi|^2 + \frac{\mu + \lam}{\rho}|\nabla^2 div\Psi|^2 \\
      &= P(t, x, y) + Q(t, x, y),
    \end{aligned}
  \end{align}
  where
  \begin{align*}
    &P(t, x, y) = - \varphi_y\Psi_{xy}\cdot\triangle\Psi + \varphi_x\nabla\varphi\cdot\nabla\triangle\varphi + \psi\nabla\Psi_y\cdot\nabla\triangle\Psi + \frac{p'(\rho)}{\rho} \nabla^2\phi\cdot\nabla\triangle\Psi \\
    &\quad + \frac{p'''(\rho)}{\gamma - 1} \phi_x\nabla\phi\cdot\nabla\triangle\varphi + \frac{\mu}{\rho^2}\triangle\varphi\nabla\phi\cdot\nabla\triangle\varphi + \frac{\mu + \lam}{\rho^2}div\Psi_x\nabla\phi\cdot\triangle\Psi_x + \frac{1}{2} \bar{u}_x |\Psi_{yy}|^2 \\
    &\quad + \bar{u}_x\Psi_x\cdot\triangle\Psi_x + \bar{u}_{xx}\varphi\triangle\varphi_x + (\frac{p'''(\rho)}{\gamma - 1} - \frac{p'''(\bar{\rho})}{\gamma - 1}) \bar{\rho}_x^2 \triangle\varphi_x + (\frac{p'(\rho)}{\rho} - \frac{p'(\bar{\rho})}{\bar{\rho}}) \bar{\rho}_{xx} \triangle\varphi_x \\
    &\quad + \frac{\mu}{\rho^2}\bar{\rho}_x\triangle\Psi\cdot\triangle\Psi_x + \frac{2\mu + \lam}{\rho^2}\bar{u}_{xx}\nabla\phi\cdot\nabla\triangle\varphi + \frac{2\mu + \lam}{\rho^2}\bar{\rho}_x\bar{u}_{xx}\triangle\varphi_x - \frac{2\mu + \lam}{\rho}\bar{u}_{xxx} \triangle\varphi_x,
  \end{align*}
  and
  \begin{align*}
    &Q(t, x, y) = \frac{1}{2}\varphi_x(|\Psi_{yy}|^2 - |\Psi_{xx}|^2) + \psi_x\nabla\varphi\cdot\nabla\triangle\psi + \nabla\psi\cdot(\varphi_y\nabla\triangle\varphi + \psi_y\nabla\triangle\psi) \\
    &\quad + \bar{u}_x\nabla\varphi\cdot\nabla\triangle\varphi + \frac{p'''(\rho)}{\gamma - 1}\phi_y\nabla\phi\cdot\nabla\triangle\psi + \frac{p'''(\rho)}{\gamma - 1}\bar{\rho}_x\nabla\phi\cdot(\triangle\Psi_x + \nabla\triangle\varphi) \\
    &\quad + \frac{\mu}{\rho^2}\triangle\psi\nabla\phi\cdot\nabla\triangle\psi + \frac{\mu + \lam}{\rho^2}\nabla\phi\cdot(div\Psi_y\triangle\Psi_y - div\Psi_x\nabla div\Psi_x - div\Psi_y\nabla div\Psi_y) \\
    &\quad + \frac{\mu + \lam}{\rho^2}\bar{\rho}_x\nabla div\Psi\cdot(\nabla\triangle\varphi - \nabla div\Psi_x) + \frac{\mu + \lam}{\rho^2}\nabla\phi\cdot(div\Psi_x\nabla\triangle\varphi + div\Psi_y\nabla\triangle\psi) \\
    &\quad + \frac{\mu + \lam}{\rho^2}\bar{\rho}_x\nabla div\Psi\cdot\triangle\Psi_x.
  \end{align*}
  Integrating the equation \eqref{27} over $[0, ~t]\times\bbr\times\bbt$ yields that
  \begin{align}
    \begin{aligned} \label{28}
      &\|\nabla^2\Psi\|^2 + \int_{0}^{t} \big[\|\bar{u}_x^{1/2}\Psi_{xx}\|^2 + \|\nabla\triangle\Psi\|^2\big] dt \\
      &\leq C \|\nabla^2\Psi_0\|^2 + C \int_{0}^{t}\int_{\bbt}\int_{\bbr} \big[|P(t, x, y)| + |Q(t, x, y)|\big]dxdydt.
    \end{aligned}
  \end{align}
  Now we estimate the terms on the right-hand side of \eqref{28}. Since the estimate for $Q(t, x, y)$ is similar to but somewhat easier than that on $P(t, x, y)$, we only consider the terms in $P(t, x, y)$ for simplicity as
  \begin{align}
    \begin{aligned} \label{29}
      &C \int_{0}^{t}\int_{\bbt}\int_{\bbr} |P(t, x, y)| dxdydt \\
      &\leq C \int_{0}^{t}\int_{\bbt}\int_{\bbr} |\varphi_y\Psi_{xy}\cdot\triangle\Psi| + |\varphi_x\nabla\varphi\cdot\nabla\triangle\varphi| + |\psi\nabla\Psi_y\cdot\nabla\triangle\Psi| + |\frac{p'(\rho)}{\rho} \nabla^2\phi\cdot\nabla\triangle\Psi| \\
      &\quad + |\frac{p'''(\rho)}{\gamma - 1} \phi_x\nabla\phi\cdot\nabla\triangle\varphi| + |\frac{\mu}{\rho^2}\triangle\varphi\nabla\phi\cdot\nabla\triangle\varphi| + |\frac{\mu + \lam}{\rho^2}div\Psi_x\nabla\phi\cdot\triangle\Psi_x| + |\frac{1}{2} \bar{u}_x |\Psi_{yy}|^2| \\
      &\quad + |\bar{u}_x\Psi_x\cdot\triangle\Psi_x| + |\bar{u}_{xx}\varphi\triangle\varphi_x| + |(\frac{p'''(\rho)}{\gamma - 1} - \frac{p'''(\bar{\rho})}{\gamma - 1}) \bar{\rho}_x^2 \triangle\varphi_x| + |(\frac{p'(\rho)}{\rho} - \frac{p'(\bar{\rho})}{\bar{\rho}}) \bar{\rho}_{xx} \triangle\varphi_x| \\
      &\quad + |\frac{\mu}{\rho^2}\bar{\rho}_x\triangle\Psi\cdot\triangle\Psi_x| + |\frac{2\mu + \lam}{\rho^2}\bar{u}_{xx}\nabla\phi\cdot\nabla\triangle\varphi| + |\frac{2\mu + \lam}{\rho^2}\bar{\rho}_x\bar{u}_{xx}\triangle\varphi_x| \\
      &\quad + |\frac{2\mu + \lam}{\rho}\bar{u}_{xxx} \triangle\varphi_x| dxdydt =\sum_{i=1}^{16} P_i.
    \end{aligned}
  \end{align}
  Each term in \eqref{29} will be estimated as follows. By Young's inequality, and interpolation inequality, it holds that
  \begin{align*}
    P_1 &= C \int_{0}^{t} \|\varphi_y\Psi_{xy}\cdot\triangle\Psi\|_{L^1} dt
        \leq \frac{1}{80} \int_{0}^{t} \|\Psi_{xy}\|^2 dt + C \int_{0}^{t} (\|\varphi_y\|_{L^4}^4 + \|\triangle\Psi\|_{L^4}^4) dt \\
        &\leq \frac{1}{80} \int_{0}^{t} \|\Psi_{xy}\|^2 dt + C \int_{0}^{t} (\|\varphi_y\|^2 \|\varphi_y\|_1^2 + \|\triangle\Psi\|^2 \|\triangle\Psi\|_1^2) dt \\
        &\leq \frac{1}{80} \int_{0}^{t} \|\Psi_{xy}\|^2 dt + CE \int_{0}^{t} (\|(\varphi_y,\nabla\varphi_y)\|^2 + \|(\triangle\Psi, \nabla\triangle\Psi)\|^2) dt.
  \end{align*}
Similarly, one can obtain
  \begin{align*}
    P_2 &= C \int_{0}^{t} \|\varphi_x\nabla\varphi\cdot\nabla\triangle\varphi\|_{L^1} dt
    \\
    &        \leq \frac{1}{80} \int_{0}^{t} \|\nabla\triangle\varphi\|^2 dt + CE \int_{0}^{t} (\|(\varphi_x,\nabla\varphi_x)\|^2 + \|(\nabla\varphi,\nabla^2\varphi)\|^2) dt,
  \end{align*}
  \begin{align*}
    P_3 = C \int_{0}^{t} \|\psi\nabla\Psi_y\cdot\nabla\triangle\Psi\|_{L^1} dt
        \leq \frac{1}{80} \int_{0}^{t} \|\nabla\triangle\Psi\|^2 dt + CE \int_{0}^{t} \|\nabla\Psi_y\|^2 dt,
  \end{align*}
 and
  \begin{align*}
    P_4 = C \int_{0}^{t} \|\frac{p'(\rho)}{\rho} \nabla^2\phi\cdot\nabla\triangle\Psi\|_{L^1} dt
        \leq \frac{1}{80} \int_{0}^{t} \|\nabla\triangle\Psi\|^2 dt + C \int_{0}^{t} \|\nabla^2\phi\|^2 dt.
  \end{align*}
Similar to $P_1$, we can obtain the estimate of $P_5-P_7$:
  \begin{align*}
    P_5 &= C \int_{0}^{t} \|\frac{p'''(\rho)}{\gamma - 1} \phi_x\nabla\phi\cdot\nabla\triangle\varphi\|_{L^1} dt\\
    &
       \leq \frac{1}{80} \int_{0}^{t} \|\nabla\triangle\varphi\|^2 dt + CE \int_{0}^{t} (\|(\phi_x,\nabla\phi_x)\|^2 + \|(\nabla\phi,\nabla^2\phi)\|^2) dt,
  \end{align*}
  \begin{align*}
    P_6 &= C \int_{0}^{t} \|\frac{\mu}{\rho^2}\triangle\varphi\nabla\phi\cdot\nabla\triangle\varphi\|_{L^1} dt\\&
        \leq \frac{1}{80} \int_{0}^{t} \|\nabla\triangle\varphi\|^2 dt + CE \int_{0}^{t} (\|(\triangle\varphi,\nabla\triangle\varphi)\|^2 + \|(\nabla\phi,\nabla^2\phi)\|^2) dt,
  \end{align*}
  and
  \begin{align*}
    P_7 &= C \int_{0}^{t} \|\frac{\mu + \lam}{\rho^2}div\Psi_x\nabla\phi\cdot\triangle\Psi_x\|_{L^1} dt \\
        &\leq \frac{1}{80} \int_{0}^{t} \|\triangle\Psi_x\|^2 dt + CE \int_{0}^{t} (\|(div\Psi_x,\nabla div\Psi_x)\|^2 + \|(\nabla\phi,\nabla^2\phi)\|^2) dt.
  \end{align*}
  By Lemma \ref{lemma2.2}, one has
  \begin{align*}
    P_8 = C \int_{0}^{t} \|\frac{1}{2} \bar{u}_x |\Psi_{yy}|^2\|_{L^1} dt
       \leq C \alpha \int_{0}^{t} \|\Psi_{yy}\|^2 dt.
  \end{align*}
  It follows from Cauchy's inequality, interpolation inequality and Lemma \ref{lemma2.2} that
  \begin{align*}
    P_9 &= C \int_{0}^{t} \|\bar{u}_x\Psi_x\cdot\triangle\Psi_x\|_{L^1} dt
       \leq \frac{1}{80} \int_{0}^{t} \|\triangle\Psi_x\|^2 dt + C \int_{0}^{t} \|\bar{u}_x\Psi_x\|^2 dt \\
       &\leq \frac{1}{80} \int_{0}^{t} \|\triangle\Psi_x\|^2 dt + C \alpha \int_{0}^{t} \|\Psi_x\|^2 dt,
  \end{align*}
  \begin{align*}
    P_{10} &= C \int_{0}^{t} \|\bar{u}_{xx}\varphi\triangle\varphi_x\|_{L^1} dt
           \leq \frac{1}{80} \int_{0}^{t} \|\triangle\varphi_x\|^2 dt + C \int_{0}^{t} \|\bar{u}_{xx}\|_{L^\infty}^2  \|\varphi\|^2 dt \\
           &\leq \frac{1}{80} \int_{0}^{t} \|\triangle\varphi_x\|^2 dt + C\alpha^{1/2} \sup_{0\leq t\leq T} \|\varphi\|^2 ,
  \end{align*}
  \begin{align*}
    P_{11} &= C \int_{0}^{t} \|(\frac{p'''(\rho)}{\gamma - 1} - \frac{p'''(\bar{\rho})}{\gamma - 1}) \bar{\rho}_x^2 \triangle\varphi_x\|_{L^1} dt
           \leq \frac{1}{80} \int_{0}^{t} \|\triangle\varphi_x\|^2 dt + C \int_{0}^{t} \|\bar{\rho}_x\|_{L^4}^4 dt \\
           &\leq \frac{1}{80} \int_{0}^{t} \|\triangle\varphi_x\|^2 dt + C \alpha^2,
  \end{align*}
  \begin{align*}
    P_{12} &= C \int_{0}^{t} \|(\frac{p'(\rho)}{\rho} - \frac{p'(\bar{\rho})}{\bar{\rho}}) \bar{\rho}_{xx} \triangle\varphi_x\|_{L^1} dt
           \leq \frac{1}{80} \int_{0}^{t} \|\triangle\varphi_x\|^2 dt + C \int_{0}^{t} \|\bar{\rho}_{xx}\|^2 dt \\
           &\leq \frac{1}{80} \int_{0}^{t} \|\triangle\varphi_x\|^2 dt + C \alpha^{2/3},
  \end{align*}
  and
  \begin{align*}
    P_{13} = C \int_{0}^{t} \|\frac{\mu}{\rho^2}\bar{\rho}_x\triangle\Psi\cdot\triangle\Psi_x\|_{L^1} dt
          \leq \frac{1}{80} \int_{0}^{t} \|\triangle\Psi_x\|^2 dt + C \alpha \int_{0}^{t} \|\triangle\Psi\|^2 dt.
  \end{align*}
Similar to $P_{10}$, we have
  \begin{align*}
    P_{14} = C \int_{0}^{t} \|\frac{2\mu + \lam}{\rho^2}\bar{u}_{xx}\nabla\phi\cdot\nabla\triangle\varphi\|_{L^1} dt
          \leq \frac{1}{80} \int_{0}^{t} \|\nabla\triangle\varphi\|^2 dt + C\alpha \int_{0}^{t} \|\nabla\phi\|^2 dt.
  \end{align*}
  By Cauchy's inequality and Lemma \ref{lemma2.2}, one can obtain
  \begin{align*}
    P_{15} &= C \int_{0}^{t} \|\frac{2\mu + \lam}{\rho^2}\bar{\rho}_x\bar{u}_{xx}\triangle\varphi_x\|_{L^1} dt
          \leq \frac{1}{80} \int_{0}^{t} \|\bar{\rho}_x\triangle\varphi_x\|^2 dt + C \int_{0}^{t} \|\bar{u}_{xx}\|^2 dt \\
           &\leq C \alpha \int_{0}^{t} \|\triangle\varphi_x\|^2 dt + C \alpha^{2/3},
  \end{align*}
  and
  \begin{align*}
    P_{16} &= C \int_{0}^{t} \|\frac{2\mu + \lam}{\rho}\bar{u}_{xxx} \triangle\varphi_x\|_{L^1} dt
           \leq \frac{1}{80} \int_{0}^{t} \|\triangle\varphi_x\|^2 dt + C \int_{0}^{t} \|\bar{u}_{xxx}\|^2 dt \\
           &\leq \frac{1}{80} \int_{0}^{t} \|\triangle\varphi_x\|^2 dt + C \alpha^{2/3}.
  \end{align*}
  Substituting the estimates of $P(t, x, y)$ and $Q(t, x, y)$ into \eqref{28} and using the Elliptic estimates $\|\triangle\Psi\| \sim \|\nabla^2\Psi\|$ and $\|\nabla\triangle\Psi\| \sim \|\nabla^3\Psi\|$, it holds that
  \begin{align}
    \begin{aligned} \label{210}
      &\|\nabla^2\Psi\|^2(t) + \int_{0}^{t}\big[ \|\bar{u}_x^{1/2}\Psi_{xx}\|^2 + \|\nabla^3\Psi\|^2\big] dt \\
      &\leq C \big[\|\nabla^2\Psi_0\|^2 + \alpha^{2/3}\big]  + C \int_{0}^{t} \|(\nabla^2\phi, \nabla^2\Psi)\|^2 dt \\
      &\quad + C(E + \alpha) \int_{0}^{t}\Big[\|(\nabla\phi, \nabla\Psi)\|_1^2 + \|\nabla^3\Psi\|^2\Big]dt.
    \end{aligned}
  \end{align}
  Combining \eqref{LEM4.4} and \eqref{210}, we complete the proof of Lemma \ref{lemma4.5}.
\end{proof}
Therefore if we take $E$ and $\alpha$ suitably small, saying $E + \alpha \leq \epsilon_1$, we have the desired a priori estimate \eqref{PRO3.2}. Thus the proof of a priori estimates in Proposition \ref{proposition3.2} is completed.

%\appendix

\end{document}